\documentclass[11pt,letterpaper]{amsart}
\usepackage{hyperref}
\usepackage{enumitem}
\usepackage{amsmath}
\usepackage{amssymb}
\usepackage{amsthm}
\usepackage{amsrefs}
\usepackage{amsfonts}
\usepackage[dvipsnames]{xcolor}
\usepackage{mathtools}
\usepackage{stackengine}

\makeatletter
\@namedef{subjclassname@2020}{%
  \textup{2020} MSC}
\makeatother

\hypersetup{
	colorlinks   = true, 
	urlcolor     = blue, 
	linkcolor    = purple, 
	citecolor   = blue 
}

\def\XXint#1#2#3{{\setbox0=\hbox{$#1{#2#3}{\int}$ }
\vcenter{\hbox{$#2#3$ }}\kern-.6\wd0}}

\makeatletter
\newcommand*{\rom}[1]{\expandafter\@slowromancap\romannumeral #1@}

\newcommand{\SL}{\mathrm{SL}^\pm}
\newcommand{\sL}{\mathfrak{sl}}

\newcommand{\M}{\mathrm{M}}
\newcommand{\X}{\mathcal{X}}

\newcommand{\R}{\mathbb{R}}
\newcommand{\Q}{\mathbb{Q}}
\newcommand{\T}{\mathbb{T}}
\newcommand{\e}{\varepsilon}
\newcommand{\cK}{\mathcal{K}}
\newcommand{\QQ}{\mathcal{Q}}

\newcommand{\Z}{\mathbb{Z}}

\newcommand{\C}{\mathcal{C}}
\newcommand{\Complex}{\mathbb{C}}
\newcommand{\cB}{\mathcal{B}}

\newcommand{\N}{\mathbb{N}}

\newcommand{\sm}{\setminus}

\newcommand{\bthm}{\begin{thm}}
\newcommand{\ethm}{\end{thm}}
\newcommand{\bproof}{\begin{proof}}
\newcommand{\eproof}{\end{proof}}
\newcommand{\blem}{\begin{lem}}
\newcommand{\elem}{\end{lem}}
\newcommand{\brem}{\begin{rem}}
\newcommand{\erem}{\end{rem}}
\newcommand{\eeqn}{\end{equation}}
\newcommand{\eeqnn}{\end{equation*}}
\newcommand{\beqn}{\begin{equation}}
\newcommand{\beqnn}{\begin{equation*}}
\newcommand{\eprop}{\end{prop}}
\newcommand{\eexm}{\end{exm}}
\newcommand{\enexm}{\end{nexm}}
\newcommand{\ecor}{\end{cor}}
\newcommand{\bcor}{\begin{cor}}
\newcommand{\bexm}{\begin{exm}}
\newcommand{\bnexm}{\begin{nexm}}
\newcommand{\bprop}{\begin{prop}}
\newcommand{\bdefn}{\begin{defn}}
\newcommand{\edefn}{\end{defn}}
\newcommand{\benum}{\begin{enumerate}}
\newcommand{\eenum}{\end{enumerate}}

\newcommand{\Mat}{\mathrm{M}_{m \times n}(\R)}
\newcommand{\Bad}{\mathrm{Bad}(m \times n)}
\newcommand{\Badb}{\mathrm{Bad}^{b}(m \times n)}
\newcommand{\Badzero}{\mathrm{Bad}^{zero}(m \times n)}
\newcommand{\BadA}{\mathrm{Bad}_{A}(m \times n)}
\newcommand{\Dirichlet}{\mathrm{Dirichlet}(m \times n)}
\newcommand{\DirichletA}{\mathrm{Dirichlet}_A(m \times n)}
\newcommand{\Dirichletb}{\mathrm{Dirichlet}^b(m \times n)}
\newcommand{\Dirichletzero}{\mathrm{Dirichlet}^{zero}(m \times n)}
\newcommand{\GenA}{\mathrm{Generic}_A(m \times n)}
\newcommand{\Genb}{\mathrm{Generic}^b(m \times n)}
\newcommand{\Gen}{\mathrm{Generic}}

\newcommand{\ZZ}{{\mathcal Z}}
\newcommand{\MM}{{\mathcal{M}}}
\newcommand{\Gr}{{\mathrm{Gr}}}
\newcommand{\EE}{\mathbb{E}}
\newcommand{\av}{\xi} 
\newcommand{\Prob}{\mathrm{Prob}}
\newcommand{\supp}{\operatorname{supp}}

\title[Non expanding random walks]{Non-Expanding Random walks on Homogeneous spaces and Diophantine approximation}

\begin{document}
\theoremstyle{plain}
\newtheorem{thm}{Theorem}[section]
\newtheorem{lem}[thm]{Lemma}
\newtheorem{prop}[thm]{Proposition}
\newtheorem{cor}[thm]{Corollary}

\theoremstyle{definition}
\newtheorem{defn}[thm]{Definition}
\newtheorem{exm}[thm]{Example}
\newtheorem{nexm}[thm]{Non Example}
\newtheorem{prob}[thm]{Problem}

\theoremstyle{remark}
\newtheorem{rem}[thm]{Remark}

\author{Gaurav Aggarwal}
\address{\textbf{Gaurav Aggarwal} \\
School of Mathematics,
Tata Institute of Fundamental Research, Mumbai, India 400005}
\email{gaurav@math.tifr.res.in}

\author{Anish Ghosh}
\address{\textbf{Anish Ghosh} \\
School of Mathematics,
Tata Institute of Fundamental Research, Mumbai, India 400005}
\email{ghosh@math.tifr.res.in}

\date{}

\thanks{ A.\ G.\ gratefully acknowledges support from a grant from the Infosys foundation to the Infosys Chandrasekharan Random Geometry Centre. G. \ A.\ and  A.\ G.\ gratefully acknowledge a grant from the Department of Atomic Energy, Government of India, under project $12-R\&D-TFR-5.01-0500$. }


\begin{abstract}
We study non-expanding random walks on the space of affine lattices and establish a new classification theorem for stationary measures. Further, we prove a theorem that relates the genericity with respect to these random walks to Birkhoff genericity. Finally, we apply these theorems to obtain several results in inhomogeneous Diophantine approximation, especially on fractals.

\end{abstract}

\maketitle

\tableofcontents

\section{Introduction}
In this paper, we establish a new classification theorem for stationary measures on the space of affine lattices. We next prove a dynamical result relating ``random genericity" of random walks, and Birkhoff genericity of diagonal flows, which we believe is interesting in its own right and which we use to partially answer a question of Solan and Wieser. Finally, we use our classification and dynamical results in turn to study inhomogeneous Diophantine approximation on fractals.\\ 

Beginning with the groundbreaking work of Bourgain, Furman, Lindenstrauss and Mozes \cite{BFLM} and Benoist and Quint \cite{BQ1, BQbook}, several important developments have occurred in the study of random walks and stationary measures on homogeneous spaces. We refer the reader in particular to the landmark works of Eskin and Mirzakhani \cite{EskinMirzakhani}, Eskin and Lindenstrauss \cite{EskLin1, EskLin2} and Simmons and Weiss \cite{SimmonsWeiss}. Further recent progress was made by Prohaska and Shi  \cite{ProhaskaSert},  Prohaska, Sert and Shi \cite{ProhaskaSertShi}, and B\'{e}nard and deSaxc\'{e} \cite{BenarddeSaxce}, Sargent and Shapira \cite{SargentShapira}, Gorodnik, Li and Sert \cite{GorodnikLiSert} among others. The main novelty of our work is that this is the first paper where a random walk with a contracting direction is discussed. Thus, it doesn't fit into the framework of any earlier work, where the \emph{expanding property} is the basic assumption of the random walk being studied. Consequently, many new ideas are needed to deal with the problem. This constitutes the first part of the paper.\\

The second part of the paper discusses the application of the above classification theorem to homogeneous dynamics. An essential ingredient to derive the latter from the former involves studying the relation between the equidistribution of a random walk starting from a point with the Birkhoff genericity of the point. The latter is a useful and well studied property in dynamics; a definition follows below. Theorem \ref{Random Genericity} essentially states that both these properties are equivalent. This partially answers a question posed in \cite{solanwieser}.\\ 

In the third part of the paper, we present several applications of Theorems \ref{Main Random walk} and \ref{Random Genericity} to inhomogeneous Diophantine approximation, among other things, we answer a question posed in \cite{MRS}. Simmons and Weiss \cite{SimmonsWeiss} were the first to realize that random walks can be profitably employed to study generic Diophantine properties of a wide class of measures, for instance, the authors were able to show that the set of badly approximable numbers are given zero measure by a wide class of self similar measures, thereby improving the work of Einsiedler, Fishman and Shapira \cite{EinsiedlerFishmanShapira} who established this fact for the middle third Cantor set. We are able to provide a comprehensive bouquet of applications of our theorems to inhomogeneous Diophantine approximation, including the more difficult `singly metric' cases where one of the variables is fixed. Our results are new even for the middle third Cantor set.\\

We now describe our results in detail.

\subsection*{Random Walks} 
Fix $m,n \geq 1$ and let $\SL_{m+n}(\R)$ denote the $(m+n) \times (m+n)$ matrices of determinant $\pm 1$. We set $G= \SL_{m+n}(\R) \ltimes \R^{m + n}$, $\Gamma= \SL_{m + n}(\Z) \ltimes \Z^{m +n}$ and denote by $\X$ the finite volume quotient $G/\Gamma$. This quotient admits a natural description as the space of \emph{affine} lattices, i.e. lattices in $\R^{m+n}$ accompanied by a shift. We also define $G'= \SL_{m + n}(\R)$, $\Gamma'= \SL_{m +n}(\R)$ and denote by $\X'$ the finite volume quotient space $G'/ \Gamma'$. This quotient admits a natural description as the space of unimodular lattices in $\R^{m+n}$ and $\X$ is a torus bundle over $\X'$. Let $\mu_\X$ (resp. $\mu_{\X'}$) denote the unique $G$ (resp. $G'$) invariant probability measure on $\X$ (resp. $\X'$). We denote by $\pi: \X \rightarrow \X'$ the natural projection. It is easy to see that $\pi_*(\mu_{\X})= \mu_{\X'}$. Also, denote by $\pi_G: G \rightarrow G'$ the natural projection.\\

Let $\nu$ be a measure on $G$ supported on a compact set $E$, such that for all $e \in E$, we have
\begin{align}
\label{eq: def e}
    e^{-1}= \left(\begin{pmatrix}
        \rho_e O_e \\ & \rho_{e}^{-m/n} I_{n}
    \end{pmatrix}, \begin{pmatrix}
         -w_e \\ 0
    \end{pmatrix} \right),
\end{align}
where $\rho_e <1$ and $O_e \in O(m)$ and $w_e \in \R^m$. Here, $O(m)$ denotes the orthogonal group with respect to some fixed chosen norm. Also, assume that $E$ is not virtually contained in $\SL_{m+n}(\R)$, i.e. there does not exist $g \in G$ such that $gEg^{-1} \subset \SL_{m+n}(\R) \ltimes \{0\}$.\\

The first main theorem of this paper is:

\begin{thm}
\label{Main Random walk}
    Let $\nu$ be a measure on $G$ as above. Then any probability measure $\mu$ on $\X$ which is $\nu$-stationary and satisfies $\pi_* \mu= \mu_{\X'}$, must equal $\mu_{\X}$.
\end{thm}

\begin{rem}
    The condition that $\nu$ is not virtually contained in $\SL_{m+n}(\R)$ is minimal, since there are measures on $\X$ which are $\SL_{m+n}(\R)$-invariant besides $\mu_{\X}$.
\end{rem}

\begin{rem}
    Theorem \ref{Main Random walk} also holds in more general cases when the elements $e^{-1}$ have $O_e' \in O(n)$ in place of $I_n$, where $O(n)$ denotes the orthogonal group w.r.t. some fixed norm on $\R^n$.
\end{rem}

\begin{rem}
    
Without the condition that $\pi_* \mu= \mu_{\X'}$, the classification of $\nu$-stationary measures implies the classification of all $a_t$-invariant measures on $\X'$ ($a_t$ is defined in (\ref{def:Diag}) below), to which we do not expect any interesting solution.
\end{rem}

\begin{rem}
    As stated earlier, the random walk given by $\nu$ differs from those in \cite{BQ1}, \cite{SimmonsWeiss} and \cite{EskLin1, EskLin2} in the fact that it is not expanding on average, nor does it satisfy any bounceback condition. This arises because the random walk has a contracting direction. To our knowledge, such a random walk on a homogeneous space has not been studied before. As a result, the exponential drift argument given here requires completely new and different ideas from the above papers.   
    In order to compensate for the existence of the contracting direction, the condition $\pi_* \mu= \mu_{\X'}$ is used in the Theorem.
\end{rem}

\subsection*{Random Genericity}
We now state the results that make up the second part of the article.

For $t \in \R$, define
\begin{equation}\label{def:Diag}
a_t= \begin{pmatrix}
    e^{t}I_m \\ & e^{-mt/n}I_n
\end{pmatrix}.
\end{equation}

\begin{defn}
\label{def: Birkhoff genericity}
    We will say that $x \in \X'$ (resp. $\X$) is Birkhoff generic if
    $$
    \lim_{T \rightarrow \infty} \frac{1}{T} \int_{[0,T]} \delta_{a_tx}  \, dt= \mu_{\X'} \text{ (resp $\mu_{\X})$}.
    $$
\end{defn}

We denote by $O(m)$ and $O(n)$ the orthogonal groups with respect to some fixed norm on $\R^m$ and $\R^n$ respectively. The second main theorem of this paper is  

\begin{thm}
\label{Random Genericity}
Suppose $\nu'$ is a measure on $G'$ supported on a compact set $E'$ such that for all $e \in E'$, we have
\begin{align}
    e^{-1}= \begin{pmatrix}
        \rho_e O_e \\ & \rho_{e}^{-m/n} O_e'
    \end{pmatrix},
\end{align}
where $\rho_e <1$, $O_e \in O(m)$ and $O_e' \in O(n)$.

Then a point $x \in \X'$ is Birkhoff generic if and only if for $(\nu')^{\otimes \N}$ almost every $(b_p)_{p \in \N} \in (G')^{ \N} $, we have that
\begin{align}
\label{eq:z2}
    \frac{1}{N} \sum_{p =1}^{N} \delta_{b_p \cdots b_1 x} \rightarrow \mu_{\X'},
\end{align}
as $N \rightarrow \infty$.
\end{thm}

\begin{rem}
    It is well known that $x \in \X'$ is Birkhoff generic if and only if for any $q>0$, the orbit $\{(a_{q})^p x\}_{p \in \N}$ equidistributes in $\X'$ w.r.t. $\mu_{\X'}$. This implies that the equidistribution of orbit under flow $\{a_t\}_{t \in \R_{\geq 0}}$ is equivalent to the equidistribution of deterministic walk $y \mapsto a_q y$. Theorem \ref{Random Genericity} extends this equivalence to the equivalence of the equidistribution of the orbit under the flow $\{a_t\}_{t \in \R}$ and the equidistribution of the random walk, given by $y \mapsto a_{\log \rho_e} y$ with probability given by $\nu$ on $E$. A statement analogous to the `if' direction of the above theorem, was proved in \cite{SimmonsWeiss}.

   One particular example is when the next step of random walk from $y$ is given by flipping an unbiased coin and then moving to $a_{2}y$ if Head appears and otherwise to $a_{3}y$ if Tails appears. Theorem \ref{Random Genericity} tells us that this random walk starting from $x$ equidistributes w.r.t. $\mu_{\X'}$ (random generic) if and only if $x$ is Birkhoff generic. 
\end{rem}
    The theorem although stated for particular case of the system $(\X', \mu_{\X'}, a_t)$, holds in much more general setting. In particular, the proof of Theorem \ref{Random Genericity} easily yields the following result.

    \begin{thm}
    Let $X$ be a locally compact, second countable, homogeneous space, $H$ a locally compact second countable group acting continuously on $X$ and $m_X$ be an $H$-invariant probability measure on $X$ such that the action of $H$ on $X$ is mixing. Let $\{a_t: t \in \R\}$ be a one parameter subgroup of $H$ such that $a_t \rightarrow \infty$ as $t \rightarrow \infty$ (i.e., leaves all compact subsets of $H$). Let $K$ be a compact subgroup of $H$ commuting with $\{a_t: t \in \R\}$ and let $\nu$ be a measure on $H$ supported on elements of the form $\{a_tk : t \in [c,d], k \in K\}$ for some $[c,d] \subset (0, \infty)$. Then, for any $x \in X$, we have that
     $$
    \lim_{T \rightarrow \infty} \frac{1}{T} \int_{[0,T]} \delta_{a_tx}  \, dt= m_{X},
    $$
    holds if and only if for $\nu^{\otimes \N}$ almost every $(b_p)_{p \in \N} \in H^{ \N} $, we have that
\begin{align}
    \frac{1}{N} \sum_{p =1}^{N} \delta_{b_p \cdots b_1 x} \rightarrow \mu_{\X'},
\end{align}
as $N \rightarrow \infty$. 
\end{thm}
One particularly interesting example is when $a_t$ is the unipotent flow on $\X'$.

We now explain how Theorem \ref{Random Genericity} partially answers a question posed by Solan and Wieser in \cite[Prob. 1.5]{solanwieser}.  

\begin{prob}[{\cite[Prob.~1.5]{solanwieser}}]
\label{prob: solanwieser Prob}
    Let $\{t_p\}$ be a sequence of integers. Determine sufficient conditions on a curve $\psi: [0,1] \rightarrow \Mat$ and the sequence $\{t_p\}$ so that for all $x \in \X'$ and for almost every $s \in [0,1]$, we have 
    \begin{align}
    \label{eq: solanwieser Prob}
        \frac{1}{N} \sum_{p=1}^N \delta_{a_{t_p} u(\phi(s))x} \rightarrow \mu_{\X'},
    \end{align}
    where for any $A \in \Mat$, $u(A)= \begin{pmatrix}
        I_m & A \\ & I_n
    \end{pmatrix}$.
\end{prob}

Suppose that $\psi: [0,1] \rightarrow \Mat$ satisfies that for all $x \in \X'$ and for a.e. $s \in [0,1]$, we have $u(\psi(s))x$ is Birkhoff generic. This happens when $\psi$ satisfies some {\em non-degeneracy} conditions mentioned in Theorem \ref{thm: genericity of Type 2 measures}. Theorem \ref{Random Genericity} tells us that as long as the elements $\{t_p- t_{p-1}\}_{p \geq 2}$ are independent identically random variables in some bounded interval $[c,d] \subset (0,\infty)$, the \eqref{eq: solanwieser Prob} holds.

\begin{rem}
    See Prop. \ref{prop: Sparse Equidistribution} for a statement about the general sequence $(t_p)_p$, for which elements $\{t_p- t_{p-1}\}_{p \geq 2}$ may not be independent identically random variables but are bounded in $(0, \infty)$.
\end{rem}

Combining Theorems \ref{Main Random walk} and \ref{Random Genericity}, we get the following application to homogeneous dynamics:

\begin{thm}
    \label{Intermediate Theorem}
    Suppose $x \in \X$ is such that $\pi(x)$ is Birkhoff generic. Let $\alpha$ be a measure on $\R^m$ of {\em Type 1} (defined in Section \ref{subsec:Type 1 measures}). Then, for $\alpha$-a.e. $v \in \R^m$, we have that $\left[I_{m+n},\begin{pmatrix}
        v \\ 0
    \end{pmatrix}\right]x$  is Birkhoff generic.
\end{thm}

\begin{rem}
    The exact definition of {\em Type 1} measures will appear in Section \ref{subsec:Type 1 measures}. Some interesting examples of {\em Type 1} measures include: normalised Lebesgue measure on a line segment (e.g. $[0,1] \times \{0\}^{m-1}$ or $\{(t,\ldots, t): t \in [0,1]\}$ etc.) or $s$-dimensional Hausdorff measure on the limit set of a finite system of compacting similarity maps satisfying an open set condition, with $s$ equalling the Hausdorff dimension of the fractal. Note that the fractals are allowed to be contained in an affine subspace of $\R^m$. 
\end{rem}

\subsection*{Diophantine approximation}

We now explain applications of our dynamical results to Diophantine approximation. Let $\Mat$ denote the space of $m \times n$ matrices with real entries. Following the notation of \cite{EinTse}, we denote by $\Bad$ the set of badly approximable systems of affine forms, namely the set of $(A, b) \in \Mat \times \R^{m}$ for which there exists $c(A, b) > 0$ such that

$$ \|Aq + b + p \| \geq \frac{c(A, b)}{\|q\|^{n/m}} $$

for all $p \in \Z^m$ and non-zero $q \in \Z^n$. Further, for fixed $b$, we set 
$$\Badb := \{A \in \Mat~:~(A, b) \in \Bad\},$$
and for fixed $A$, we set 
$$\BadA := \{b \in \R^m~:~(A, b) \in \Bad\}.$$ 

In \cite{Kleinbock}, it is shown that $\Bad$ has full Hausdorff dimension. In the case where $b = 0$, namely the homogeneous case,  this was further generalized in \cite{KleinbockWeiss} to the statement that $K \cap Bad$ is winning (for a suitable variant of the Schmidt game) for sets that appear as supports of a wide class of measures, called friendly measures, which were previously introduced in \cite{KLW}. In \cite{EinTse}, Einsiedler and Tseng, answering a question posed by Kleinbock, studied the more difficult case when one of the parameters $A$ or $b$ is fixed. That is, it was proved that $K \cap \Badb $ and $K \cap \BadA$ are both winning for a wide class of $K$. We refer the reader to \cite{EinTse} for the precise statements. The results of \cite{EinTse} leave open the natural question of the measure of these sets with respect to the natural measure supported on a fractal. There are three natural questions one can ask for a suitably wide class of fractals $K$:

\begin{enumerate}
\item Does $\Bad \cap K $ have zero measure?\\
\item Does $\Badb \cap K$ have zero measure?\\
\item Does $\BadA \cap K$ have zero measure?
\end{enumerate}

We address these questions in this paper. Somewhat paradoxically, it turns out that the measure theoretic question is often more nuanced than the Hausdorff dimension question. It is relatively easy to construct fractals that entirely consist of badly approximable numbers.\\

We now turn to the opposite end of the Diophantine spectrum in a certain sense.

\begin{defn}
\label{def:Dirichlet}
    An affine form $(A,b)$ is said to be Dirichlet improvable if there exists a $ \lambda(A,b) >0$ such that for all sufficiently large $Q \geq 1$, there exists $q \in \Z^n \setminus \{0\}$ and $p \in \Z^m$ such that 
    $$\|q\| \leq Q \text{ and } \|A q + p + b\| \leq \lambda(A,b)Q^{-n/m}.$$
    We define $\Dirichlet$ as the set of all Dirichlet improvable linear forms. Furthermore, for a fixed $b$, we set 
    $$
    \Dirichletb:= \{A: (A,b) \in \Dirichlet\},
    $$
    and for a fixed $A$, we set
    $$
    \DirichletA:= \{b : (A,b) \in \Dirichlet\}.
    $$
\end{defn}
These sets have been extensively studied in recent years, we refer the reader in particular to \cite{KleinbockWadleigh}.
\begin{defn}
    \label{def: Generic type}
    An affine form $(A,b)$ is said to be of generic type if the element
    $$
    \left[\begin{pmatrix}
        I_m & A \\ &I_n
    \end{pmatrix}, \begin{pmatrix}
        b \\ 0
    \end{pmatrix} \right] \Gamma
    $$
    is Birkhoff generic. We define $\Gen \subset \M_{m \times n}(\R) \times \R^m$ as the set of all affine forms that are of generic type. Furthermore, for a fixed $b$, we set 
    $$
    \Genb:= \{A: (A,b) \in \Gen\},
    $$
    and for a fixed $A$, we set
    $$
    \GenA:= \{b : (A,b) \in \Gen\}.
    $$
\end{defn}

Using Theorem \ref{Intermediate Theorem}, we get the following result:
\begin{thm}
\label{thm:Single Inhomogeneous case 1}
    There exists a subset $\mathcal{M} \subset \M_{m \times n}(\R)$, which is large in the sense that it gets full measure w.r.t. each measure of {\em Type 2} (defined in Section \ref{subsec: Type 2 measures}) and {\em Type 3} (defined in Section \ref{subsec: Type 3 measures}), such that for every $A \in \mathcal{M}$ the following holds: for every measure $\alpha$ of {\em Type 1} (defined in Section \ref{subsec:Type 1 measures}), we have:
    $$\alpha(\BadA)= \alpha(\DirichletA)= 0 \text{ and } \alpha(\GenA)=1.$$
    In fact, the set $\mathcal{M}$ is maximal, in the sense that if $A \notin \mathcal{M}$, then $\GenA = \emptyset$.
\end{thm}

\begin{rem}
    The exact definition of {\em Type 2} measures will appear in Section \ref{subsec: Type 2 measures}. In case $m=1$, these include the pushforward of Lebesgue measure $[0,1]$ under an analytic map such that the image is not contained in any affine subspace. Examples include the famous Veronese curve $t \mapsto (t, t^2, \ldots, t^n)$. 
\end{rem}
\begin{rem}
    The exact definition of {\em Type 3} measures will appear in Section \ref{subsec: Type 3 measures}. In case $n=1$ or $m=1$, the examples include $s$-dimensional Hausdorff measure on the limit set of a finite system of compact similarity maps satisfying an open condition, with $s$ equals the Hausdorff dimension of the fractal. In this case, the fractals should not be contained in any affine subspace of $\R^n$ or $\R^m$ respectively. 
\end{rem}

\begin{rem}
In \cite{MRS}, Moshchevitin, Rao, and Shapira prove that under certain explicit conditions on $A$, for any nontrivial algebraic measure (see loc. cit. for the definition) $\mu$ on the $m$ torus $\mathbb{T}^m$, $\mu(\BadA) = 0$. They further ask if such a result can be proven for other classes of measures. Theorem \ref{thm:Single Inhomogeneous case 1} provides an answer to this question. We note that this result for the Lebesgue measure on the torus was known earlier. In dimension $1$, due to Kim \cite{Kim}, and in higher dimension due to Shapira \cite{Shapira}. See also \cite{BDGW23a, Kim23, Mosh}.
\end{rem}

\begin{rem}
    Theorem \ref{thm:Single Inhomogeneous case 1} sets a dichotomy: for any $A \in \Mat$, either $\GenA$ is an empty set or it is large, in light of the fact that it gets full measure w.r.t. all measures of {\em Type 1} (defined in Section \ref{subsec:Type 1 measures}).
\end{rem}

Also using the results of \cite{ProhaskaSertShi} in place of Theorem \ref{Intermediate Theorem}, we get the following result:
\begin{thm}
\label{thm:Single Inhomogeneous case 2}
    For every $b \in \R^m$ and for every measure $\gamma$ of {\em Type $3'$} (defined in Section \ref{subsec: Type 3 measures},) we have $$\gamma(\Badb)= \gamma(\Dirichletb)= 0.$$ Moreover, if $b \notin \Q^m$, we have $\gamma(\Genb)=1$. 
\end{thm}

\begin{rem}
    It turns out that even though the questions involved in Theorem \ref{thm:Single Inhomogeneous case 1} and Theorem \ref{thm:Single Inhomogeneous case 2} look very similar, the associated random walks have completely different properties. As a result, both the problems are of different levels of difficulty. For Theorem \ref{thm:Single Inhomogeneous case 2}, we can use existing results, namely the work of Eskin and Lindenstrauss \cite{EskLin1} along with the strategy of Simmons and Weiss \cite{SimmonsWeiss}, in the form of \cite[Thm.~1.10]{ProhaskaSertShi} to prove the result. Theorem \ref{thm:Single Inhomogeneous case 1}  turns out to be substantially more involved and needs the full strength of the new dynamical results proved in this paper.
\end{rem}

Combining Fubini's Theorem with Theorem \ref{thm:Single Inhomogeneous case 1} and Theorem \ref{thm:Single Inhomogeneous case 2} we get that
\begin{thm}
    \label{thm:Inhomogeneous}
    Let $\lambda$ be a probability measure on $\M_{m \times n}(\R) \times \R^m$ which is in the convex hull of measures of the form:
    \begin{enumerate}
        \item $\{\alpha \times \beta: \alpha \text{ is of \em{Type 1} and $\beta$ is of \em{ Type 3}}\}$
        \item $\{\alpha \times \beta: \alpha \text{ is of \em{Type 2} and $\beta$ is of \em{ Type 3}}\}$
        \item $\{\alpha \times \beta: \text{ $\alpha$ is of \em{ Type $3'$}}\}$.
    \end{enumerate}
    Then $\lambda(\Bad)= \lambda(\Dirichlet)= 0$.
\end{thm}

\subsection*{Previous work}
We survey the literature concerning measure theoretic Diophantine approximation on fractals briefly. Subsequent to the important work of Einsiedler, Fishman and Shapira, \cite{EinsiedlerFishmanShapira}, there was the breakthrough \cite{SimmonsWeiss} where it was demonstrated that both $\Badzero$ and $\Dirichletzero$ have zero measure for a wide class of fractals. A weighted analogue of $\Badzero$ was examined in \cite{ProhaskaSertShi} and was shown to have measure zero. The two previous works use the technology of random walks on homogeneous spaces. In \cite{BDGW23}, a very general approach using techniques from geometric measure theory and Diophantine approximation was developed to show that badly approximable sets in a wide variety of contexts have measure zero. Another important advance was made in the paper \cite{KhalilLuethi} of Khalil and Luethi  where a Khintchine theorem is proven for certain measures. Their results have subsequently been improved in certain cases by Datta and Jana \cite{DattaJana} and by B\'{e}nard, He and Zhang \cite{BenardHeZhang}. Singular vectors on fractals have been studied by Khalil \cite{Khalilsing} and further in a weighted context in \cite{aggarwalghoshpacking}. Inhomogeneous singular forms are studied in  \cite{aggarwal2025}.

\subsection{Structure of the paper}
The paper is divided into three parts. Sections 2-8 constitute the first part of the paper. The main goal of this part is to prove Theorem \ref{Main Random walk}. Sections \ref{Notation},\ref{Conditional Measures},\ref{Horocycle Flows},\ref{Disintegration of nu along stabilizers} are mostly devoted to introducing notation and borrow significantly from the work of Benoist and Quint taken from \cite{benoistquinttranslation} applied to our case. We have included them to keep the paper self-contained. Section \ref{Basic Facts and Observations} contains some key observations that are essential for the exponential drift argument in section \ref{The exponential drift}. Section \ref{The exponential drift} contains the exponential drift argument. This requires significant new ideas to deal with non-expanding random walks. The exponential drift will allow us to decompose a stationary measure $\mu$ into homogeneous measures. Section \ref{Proof of Theorem Main Random walk} deduces that such a decomposition implies that $\mu$ equals $\mu_{\X}$. In \cite{BQ1}, \cite{SimmonsWeiss}, such a deduction was proved using expanding properties of the random walk. Concretely, it was shown in \cite{BQ1} and \cite{SimmonsWeiss} that the only possible decomposition of stationary measures into homogeneous measures is the trivial decomposition. However, in our case, we cannot use any such argument. This is not only due to the fact that the random walk is not expanding but also because a non-trivial decomposition of $\mu_{\X}$ into homogeneous measures is possible. Thus a completely new argument is needed here as well.\\

Sections \ref{Proof of Theorem Random Genericity}, \ref{Application to Homogeneous Dynamics} constitute the second part of the paper. Section \ref{Proof of Theorem Random Genericity} gives a proof of Theorem \ref{Random Genericity}. This section is independent of previous sections and can be read on its own. The section \ref{Application to Homogeneous Dynamics} defines {\em Type 1} measures and provides a proof of Theorem \ref{Intermediate Theorem}. If the reader assumes Theorems \ref{Main Random walk}, \ref{Random Genericity}, then they can read this section without reading previous sections.\\

Section \ref{Diophantine Approximation} constitutes the third part of the paper. This section proves the Diophantine results mentioned above. The reader who is only interested in the Diophantine part can assume Theorem \ref{Intermediate Theorem} and can directly read this section without reading previous sections.

\vspace{1cm}
\section{Random Walks} \label{Random Walk}

\subsection{Notation}
\label{Notation}
The main goal of this part of the paper is to prove Theorem \ref{Main Random walk}. Before starting, we fix a measure $\nu$ on $G$ as in Theorem \ref{Main Random walk} and a $\nu$-stationary measure $\mu$ on $\X$. We will denote by $E$ the support of the measure $\nu$. In what follows, we will sometimes also treat $E$ as an indexing set and use the notation $g_e := e$ for $e \in E$, to denote the corresponding element in $G$. Also, we will sometimes treat $\nu$ as a measure on $E$, instead of $G$.\\

We define the following four dynamical systems required for the proof:\\

{\bf The system $(B, \cB, \beta, T)$}. Here $B= E^\N$, which is equipped with the product $\sigma$-algebra $\cB$ and measure $\beta = \nu^{\otimes \N}$. We denote by $T: B \rightarrow B$ the shift map, which clearly preserves $\beta$.\\

{\bf The system $(B^\tau, \cB^\tau, \beta^\tau, (T^\tau_l)_{l \in \R})$}. Define a map $\tau: B \rightarrow \R$ as $\tau(b)= - \log (\rho_{b_1})$, where for $e \in E$, $\rho_e$ is defined as in \eqref{eq: def e}. For $p \geq 0$, and $b \in B$, denote
\begin{align}
    \label{eq: def tau p}
    \tau_p(b) = \tau(T^p(b)) + \cdots + \tau(b).
\end{align}
We denote by $(B^\tau, \cB^\tau, \beta^\tau, T^\tau) $ the suspension of $(B, \cB, \beta, T)$ with roof function $\tau$, i.e., 
\begin{align}
    \label{eq: def B tau}
    B^\tau = \{c= (b,k) \in B \times \R: 0 \leq k < \tau(b)\}.
\end{align}
The measure $\beta^\tau$ is obtained by normalizing the restriction to
$B^\tau$ of the product measure of $\beta$ and the Haar measure of $\R$, the $\sigma$-algebra $\cB^\tau$ is the product $\sigma$-algebra, and for almost every $l \in \R_+$ and $c =
(b,k) \in B^\tau$, 
$$
T^\tau_l(c) = \left(T^{p_l(c)}b, k+l - \tau_{
    p_l(c)}(b) \right)
$$
where
$$
p_l(c) = \max\{ p \in \N: k+l - \tau_{ p}(b) \geq 0\}.
$$
The flow $T^\tau_l$ is then defined for all positive times and preserves the measure $\beta^\tau$ (see e.g. \cite[Lem. 2.2]{BQ1}). It is easy to see that $\beta^\tau$ is a $T_l^\tau$-ergodic measure.\\

{\bf The system $(B^\X, \cB^\X, \beta^\X, T^\X)$}.
We denote by $B^\X = B \times \X$ equipped with the product sigma-algebra $\cB^\X$. We denote by $T^\X$ the transformation on $B^\X =  B \times \X$ given by, for $(b,x) \in B^\X$,
\begin{align}
    \label{eq:def T X}
    T^\X(b,x) = (Tb, b_0^{-1}x).
\end{align}
Define the probability measure $\mu_b$ on $\X$ as
\begin{align}
    \label{def: mu b}
    \mu_b = \lim_{p \rightarrow \infty} ({b_0}_* \cdots {b_p}_*)\mu.
\end{align}
The existence of the limit in \eqref{def: mu b} for $\beta$-a.e. $b \in B$ is shown in \cite[Section~3]{BQ1}, where it is  furthermore shown that for $\beta$-a.e. $b \in B$, we have
\begin{align}
    \mu_b = b_{0*}\mu_{Tb}, \label{Invariance under T of mu b} \\
    \mu = \int_{B} \mu_{b} \, d\beta(b). \label{mu equal integral of mu b}
\end{align}
The measure $\beta^\X$ is defined as
\begin{align}
    \label{eq:def beta X}
    \beta^\X = \int_{B} \delta_b \otimes \mu_b \, d\beta(b).
\end{align}
It is easy to see that $\beta^\X$ is $T^\X$-invariant using \eqref{Invariance under T of mu b}.

{\bf The system $(B^{\tau, \X}, \cB^{\tau, \X}, \beta^{\tau, \X}, (T^{\tau, \X}_l)_{l \in \R})$}.
We denote $B^{\tau, \X} = B^\tau \times \X$ equipped with the product sigma algebra $\cB^{\tau, \X}$. For $\beta^\tau$-a.e. $c= (b,k) \in B^\tau$, we denote $\mu_c= \mu_b$, where $\mu_b$ is defined in \eqref{def: mu b}.

For $l \geq 0$ and for $\beta^\tau$-a.e. $c = (b,k) \in B^\tau$,
we introduce the map $\rho_l(c)$ of $\X$ given by, for any $x \in
\X$, 
$$
\rho_l(c)x = b^{-1}_{p_l(b,k)-1} \cdots b_0^{-1} x.
$$
We define for $(c,x) \in B^{\tau, X}$
and $l \geq 0$, 
$$
T^{\tau, X}_l (c,x) = \left( T^\tau_l c, \rho_l(c)x\right).
$$
By \cite[Lem.~3.6]{BQ1}, for $\beta^\tau$-a.e. $c = (b,k) \in B^\tau$ and for every $l
\geq 0$, one has 
$$
\mu_{T^\tau_l c} = \rho_l(c)_* \mu_c.
$$
Thus, we have that for all $l \geq 0$, the transformation $T_l^{\tau, X}$ of
$B^{\tau, X}$ preserves the measure $\beta^{\tau, X}$ (see also \cite[Lem.~3.7]{BQ1}).


\section{Basic Facts and Observations}
\label{Basic Facts and Observations}

\begin{lem}
\label{lem:Imp Project mu b equal mu}
    For $\beta$-a.e. $ b \in B$, we have
\begin{align}
    \label{eq: Imp Project mu b equal mu}
    \pi_*(\mu_b) = \mu_{\X'}.
\end{align}
\end{lem}
\begin{proof}
    Note that $\pi: \X \rightarrow \X'$ is continuous, hence we have
    $$
        \pi_* \left( \lim_{p \rightarrow \infty} ({b_0}_* \cdots {b_p}_*)\mu \right) = \lim_{p \rightarrow \infty} \pi_* \left(({b_0}_* \cdots {b_p}_*)\mu \right).
    $$ 

    Thus, we have
    \begin{align*}
        \pi_*(\mu_b) &= \pi_* \left( \lim_{p \rightarrow \infty} ({b_0}_* \cdots {b_p}_*)\mu \right)  \\
        &= \lim_{p \rightarrow \infty} \pi_* \left(({b_0}_* \cdots {b_p}_*)\mu \right) \\
        &= \lim_{p \rightarrow \infty} (\pi_{G}({b_0} \cdots {b_p}))_*  \mu_{\X'} \\
        &= \mu_{\X'},
    \end{align*}
    where the last equality follows from $G'$-invariance of $\mu_{\X'}$. 
\end{proof}

We follow the notation of \cite{BQ1} to parameterize the branches of the inverses
of $T^{\tau, \X}_l$. For $q \geq 0$ and $a,b \in B$, we denote by $a[q]$
the beginning of the word $a$ written from right to left as
$a[q]=(a_{q}, \ldots, a_1)$ and $a[q]b \in B$ the concatenated
word 
$$a[q]b=(a_{q}, \ldots, a_1, b_0, b_1, \ldots, b_p, \ldots).$$
For $c = (b,k) \in B^{\tau}$ and $l$ in $\R_+$, let $q_{l,c}:
B\to \N$, $h_{l,c}: B \rightarrow B^{\tau}$ and $h_{l, c,x}: B \to B^{\tau}$ the maps given, for $a \in B$, by 
\begin{align*}
    q_{l, c} = \widetilde q_{l, c'} \ \ \text{and} \ \ h_{l, c} =
\widetilde h_{l, c'}, \ \ \text{where} \ \ c' = T^\tau_l(c)
\end{align*}
and 
$$
\widetilde q_{l, c}(a) = \min\{q \in \N: k-l+\tau_{q}(a[q]b) \geq 0\},
$$
$$
\widetilde h_{l, c}(a) = (a[q]b, k-l+\tau_{ q}(a[q]b)) \ \ \text{with} \ \ q = \widetilde q_{l, c}(a).
$$
Since, $\tau(b) = - \log \rho_b  \geq -  \log ( \max_{e \in E} \rho_e)  > 0$, we have that the function $\widetilde q_{l, c}$ is always finite and the image of the map $\widetilde h_{l,c,x}$ is the fiber $(T^{\tau,\X}_l)^{-1}(c,x)$. The function $q_{l,c}$ is thus also always finite.

Note that the image of the map $h_{l, c,x} $ is the fiber of
$T^{\tau, \X}_l$ passing through $(c,x)$:
$$
\{(c'', x'') \in B^{\tau, \X} : T^{\tau,\X}_l(c'',x'') = T^{\tau_l,\X}(c)\}, 
$$
that is the atom of $c$ in the partition associated with the
$\sigma$-algebra $$\QQ_l^{\tau, X} := \left(T_l^{\tau, X} \right)^{-1}
(\cB^{\tau, X}).$$ 

For $c = (b,k) \in B^{\tau}$ and $l$ in $\R_+$, let us define $h_{l,c,x}: B\to B^{\tau, \X}$ as
\begin{align}
    \label{eq: def h l c x}
    h_{l,c,x}(a)= (h_{l,c}(a), \rho_l(h_{l,c}(a))^{-1} \rho_l(c)x )
\end{align}

Denote by $\QQ_\infty^{\tau, X}$ the tail
$\sigma$-algebra of $\left(B^{\tau, X}, \cB^{\tau, X}, \beta^{\tau,
    X}, T^{\tau, X} \right)$, that is the decreasing intersection of
sub-$\sigma$-algebras $\QQ_\infty^{\tau, X} = \bigcap_{l \geq 0}
\QQ_l^{\tau, X}$.

\begin{prop}[{\cite[Prop.~2.3, Cor.~3.8]{BQ1}}]
\label{cor:BQ 3.8}
For any $\beta^{\tau, X}$-integrable function $\varphi : B^{\tau, X}
\to \R$, for every $l \geq 0$, for $\beta^{\tau, X}$-a.e. $(c,x)
\in B^{\tau, X}$, one has 
\begin{equation}
    \label{eq: BQ 3.11}
    \EE \left( \varphi | \QQ^{\tau, X}_l \right) (c,x) = \int_{B}
\varphi(h_{l, c,x}(a)) d\beta(a).
\end{equation}
\end{prop}

\begin{lem}
    \label{lem: Condition on Fractal}
    There exists a $\kappa >0, T >0$ and constants $C_1, C_2,  C_3>0$ such that for all $c= (b,k) \in B^\tau$, we have that for all $l \geq T$, there exists a measurable subset $Y_l(c) \subset B$ satisfying $\beta(Y_l(c)) > \kappa$ and for all $a \in Y_l(c)$ we have
    \begin{align}
        \label{eq: Condition on Fractal}
        \rho_l(h_{l,c}(a))^{-1} \rho_l(c) = [A_{l,c,a}, v_{l,c,a}],
    \end{align}
     which satisfy 
    $$
        A_{l,c,a} \in \left\{ \begin{pmatrix}
               e^t O \\ & e^{-mt/n} I_n
            \end{pmatrix}: |t| \leq 2C_1 , O \in \text{O}_{m}(\R)\right\}
    $$
    and $v_{l,c,a} \in \R^m \times \{0\}^n$ such that $e^{-C_1}C_2 < e^{-l}\|v_{l,c,a}\|_{\infty} < e^{C_1}C_3$.
\end{lem}
\begin{proof}
    Consider the map $\eta: B \rightarrow \R^m$ defined by 
    $$
    \eta(b)= \lim_{p \rightarrow \infty} \phi_{b_1} \circ \cdots \circ \phi_{b_p}(0),
    $$
    where for each $e \in E$, $\phi_e: \R^m \rightarrow \R^m$ is defined as $\phi_e(x) = \rho_e O_e x + w_e$. Note that the map is well defined and continuous since $\max_{e \in E} \rho_e <1$.

    First of all $\cK= \eta(B)$ is not a single point, since otherwise if $\cK=\{x\}$, then $$[I_{m+n},x]^{-1} e^{-1} [I_{m+n},x] \in \SL_{m+n}(\R),$$ hence $E$ will virtually be contained in $\SL_{m+n}(\R)$. Consider the measure $\mu_\cK = \eta_* \beta$ on $\cK$. Clearly the support of the measure $\mu_\cK$ equals $\cK$ and $\cK$ is compact using the fact that $\supp(\beta) = B$ and $\eta$ is continuous.\\ 

    For $x \in \R^m$ and $s>0$, set $B_s(x)= \{y \in \R^m: \|y-x\|< s\}$. We claim that there exists a $\delta>0$ and $\kappa>0$ such that for all $x \in \cK$, we have $\mu_\cK(B_\delta(x))< 1-\kappa$. To see this, assume that $\delta$ is small enough so that $\cK$ is not contained in the closure of any ball of radius $2\delta$. Now assume by contradiction that for every $\kappa_p \rightarrow 0$, we can find $x_p$ such that $\mu_{\cK}(B_\delta(x_p))> 1- \kappa_p$. Using the fact that $\cK$ is compact and passing to a subsequence, we may assume that $x_p \rightarrow x_0$ as $p \rightarrow \infty$. Note that for all large enough $p$, we have $\mu_{\cK}(B_{2\delta}(x_0)) \geq \mu_{\cK}(B_\delta(x_p)) > 1-\kappa_p $, which implies that $\mu_{\cK}(B_{2\delta}(x_0)) = 1$. This in turn implies that $\supp(\mu_{\cK}) \subset \Bar{B_{2\delta}(x)}$. Since, $\supp(\mu_{\cK}) = \cK$. This contradicts the definition of $\delta$. Hence the claim follows.

    Let $\rho = \max_{e \in E} \rho_e$. Let $M$ be large enough so that $\cK \subset B_M(0)$ and $\phi_e(B_M(0)) \subset B_M(0)$ for all $e \in E$. Let $C_1 = -\log(\min_{e \in E}\rho_e)$. Suppose $L_1$ is chosen large enough so that $3 \rho^{L_1} M< \delta$. Choose $T$ large enough so that $ T> L_1C_1$. Let $C_3 = 2M$ and let $C_2 = \delta - 2 \rho_{L_1} M$.\\

    We claim that these constants suffice. To see this, let $c=(b,k) \in B^\tau$ and $a \in B$ be arbitrary. Then for $l \geq T$ we have 
    \begin{align*}
        &\rho_l(h_{l,c}(a))^{-1} \rho_l(c) \\
        &=  a_q \cdots a_1 b_p^{-1} \cdots b_1^{-1},
    \end{align*}
    where $p \in \N$ is given so that $p = \max\{j \in \N: k+l - \tau_j(b)>0 \}$ and $q = q_{l,c}(a)$. Note that $\tau_j(d) \leq jC_1$ for all $d \in B$, hence $p \geq L_1$ and $q \geq L_1$. Also, note that 
    \begin{align*}
        &b_p^{-1} \cdots b_1^{-1}\\
        &= \left(\begin{pmatrix}
            \rho_{b_p}O_{b_p} \\ & \rho_{b_p}^{-m/n}I_n
        \end{pmatrix}, \begin{pmatrix}
            -w_{b_p} \\ 0
        \end{pmatrix} \right) \cdots \left(\begin{pmatrix}
            \rho_{b_1}O_{b_1} \\ & \rho_{b_1}^{-m/n}I_n
        \end{pmatrix}, \begin{pmatrix}
            -w_{b_1} \\ 0
        \end{pmatrix} \right) \\
        &= \left(\begin{pmatrix}
            \rho_{b_p}\cdots \rho_{b_1} O_{b_p} \cdots O_{b_1} \\ & (\rho_{b_p}\cdots \rho_{b_1})^{-m/n}I_n
        \end{pmatrix}, \begin{pmatrix}
            -\phi_{b_p} \circ \cdots \circ \phi_{b_1}(0) \\ 0
        \end{pmatrix} \right) \\
        &= \left(I_{m+n}, \begin{pmatrix}
            -\phi_{b_p} \circ \cdots \circ \phi_{b_1}(0) \\ 0
        \end{pmatrix} \right) \left(\begin{pmatrix}
            \rho_{b_p}\cdots \rho_{b_1} O_{b_p} \cdots O_{b_1} \\ & (\rho_{b_p}\cdots \rho_{b_1})^{-m/n}I_n
        \end{pmatrix}, 0 \right).
    \end{align*}
    
    Similarly, we have
    \begin{align*}
        &a_q \cdots a_1\\
        &= (a_1^{-1} \cdots a_q^{-1})^{-1} \\
        &=  \left(\begin{pmatrix}
            \rho_{a_1}\cdots \rho_{a_q} O_{a_0} \cdots O_{a_q} \\ & (\rho_{a_0}\cdots \rho_{a_q})^{-m/n}I_n
        \end{pmatrix}, \begin{pmatrix}
            -\phi_{a_1} \circ \cdots \circ \phi_{a_q}(0) \\ 0
        \end{pmatrix} \right)^{-1} \\
        &= \left(\begin{pmatrix}
            (\rho_{a_q}\cdots \rho_{a_1})^{-1} O_{a_q}^{-1} \cdots O_{a_1}^{-1} \\ & (\rho_{a_q}\cdots \rho_{a_1})^{m/n} I_n
        \end{pmatrix}, 0 \right) \left(I_{m+n}, \begin{pmatrix}
            \phi_{a_0} \circ \cdots \circ \phi_{a_q}(0) \\ 0
        \end{pmatrix} \right).
    \end{align*}
    Thus, we have
    \begin{align*}
        &\rho_l(h_{l,c}(a))^{-1} \rho_l(c) = [A_{l,c,a}, v_{l,c,a}],
    \end{align*}
    where
    $$
    A_{l,c,a} = \begin{pmatrix}
            (\rho_{a_q}\cdots \rho_{a_1})^{-1} \rho_{b_p}\cdots \rho_{b_1}  O_{a_q}^{-1} \cdots O_{a_1}^{-1} O_{b_p} \cdots O_{b_1}  \\ & (\rho_{a_q}\cdots \rho_{a_1}(\rho_{b_p}\cdots \rho_{b_1})^{-1} )^{m/n}I_n
        \end{pmatrix}
    $$
    and
    $$
    v_{l,c,a} = (\rho_{a_q}\cdots \rho_{a_1})^{-1} O_{a_q}^{-1} \cdots O_{a_1}^{-1} \left(\phi_{a_0} \circ \cdots \circ \phi_{a_q}(0) -\phi_{b_p} \circ \cdots \circ \phi_{b_1}(0) \right).
    $$
    Note that by definition of $p$, we have that 
    $$
    e^{-l-C_1} \leq e^{-k-l} \leq \rho_{b_p}\cdots \rho_{b_1}= e^{-\tau_p(b))} \leq e^{C_1 - k-l } \leq e^{C_1-l}.
    $$
    Also, by definition of $q$, we have that
    $$
    e^{l-C_1} \leq e^{l-k} \leq (\rho_{a_q}\cdots \rho_{a_1})^{-1}= e^{\tau_q(a[q]b))} \leq e^{C_1 +l-k } \leq e^{l+ C_1}.
    $$
    Thus, $A_{l,c,a}$ satisfies the conditions of the theorem. Similarly, we have that $\|v_{l,c,a}\| \leq e^{C_1}C_3$.

    Let $x_{c,l} \in \cK$ be arbitrary such that $\|\phi_{b_p} \circ \cdots \circ \phi_{b_1}(0)- x_{c,l}\| \leq \rho^{L_1}.M$. Note that such an element exists because for any $d \in B$ with $d_i = b_{p+1-i}$ for $1 \leq i \leq p$, we have 
    $$
    \|\eta(d)- \phi_{b_p} \circ \cdots \circ \phi_{b_1}(0)\| \leq \rho^{p} \| \eta(T^p(d))\| \leq \rho^p M.
    $$
    Then, let $Y_l(c)= \eta^{-1}(\cK \setminus B_{\delta}(x_{c,l}))$. Clearly $\beta(b)> \kappa$. Let $a \in Y_l(c)$ be arbitrary. Then, we have 
    \begin{align*}
        &\|\phi_{a_0} \circ \cdots \circ \phi_{a_q}(0) -\phi_{b_p} \circ \cdots \circ \phi_{b_1}(0)\| \\
        &\geq  -\|\phi_{a_0} \circ \cdots \circ \phi_{a_q}(0) - \eta(a)\| + \|\eta(a)- x_{c,l}\| - \|x_{c,l}- \phi_{b_p} \circ \cdots \circ \phi_{b_1}(0)\| \\
        &\geq \delta- 2 \rho^{L_1}.M \geq C_2.
    \end{align*}
    Thus, inequality that $\|v_{l,c,a}\|> e^{l-C_1}C_2$ holds for all $a \in Y_l(c)$. This proves the lemma.
\end{proof}

\section{Conditional Measures}
\label{Conditional Measures}
We follow \cite{BQ1} closely in this section and are indebted to the translation \cite{benoistquinttranslation} by Barak Weiss. 
Let $R$ be a locally compact separable metrizable group and $(Z, \ZZ)$ a standard Borel space with a Borel action of $R$. Let $\lambda$ be a Borel probability measure on $Z$. Assume that the stabilizer subgroups for the action of $R$ on $Z$ are discrete. 

Let $\MM(R)$ denote the space of positive nonzero Radon measures on
$R$ and let $\MM_1(R) = \MM(R) /\simeq$ be the space of such measures
up to scaling as in \cite{BQ1}.

A Borel subset $\Sigma \subset Z$ is called a discrete
  section of the action of $R$ if, for any $z \in Z$, the set of
visit times $\{r \in R: rz \in \Sigma\}$ is discrete and closed in
$R$. By \cite{Kechris}, we know that there is a discrete
section $\Sigma$ for the action of $R$ such that $R\Sigma = Z$. 

Choose a discrete section $\Sigma$ for the action of $Z$ on $R$ and
denote $a: R \times \Sigma \to Z, (r,z) \mapsto rz$. For any positive Borel
function $f$ on $R \times \Sigma$, the following measure
$a^*\lambda$ on $R \times \Sigma$ 
\begin{equation}
    \label{eq: 4.1}
    a^*\lambda(f) = \int_Z \left(\sum_{(r,z') \in a^{-1}(z)} f(r,z')
\right) d\lambda(z),
\end{equation}
defines a $\sigma$-finite Borel measure on $R \times \Sigma$. 

Denote by $\pi_\Sigma : R\times \Sigma \to \Sigma$ the projection on
the second factor, and by $\lambda_\Sigma$ the image under
$\pi_\Sigma$ of a finite measure on $R \times \Sigma$ equivalent to
$a^*\lambda$. We therefore have, for any positive Borel function on $R
\times Z$, 
\begin{equation}
\label{eq: 4.2}
a^*\lambda(f) = \int_\Sigma \int_R f(r,z) d\sigma_\Sigma(z)(r)
d\lambda_\Sigma(z). 
\end{equation}

We denote by $t_r$ the right translation by an element $r \in R$. 

\begin{lem}[{\cite[Lem.~4.1]{BQ1}}] \label{lem: 4.1}
Let $\Sigma$ be a discrete section for the action of $R$ on $Z$. For
$\lambda_\Sigma$-a.e. $z \in \Sigma$, for all $r \in R$ such that $rz
\in \Sigma$, we have
$$
\sigma_\Sigma(z) \simeq t_{r*} \sigma_\Sigma(rz).
$$
\end{lem}

\begin{prop}[{\cite[Prop.~4.2]{BQ1}}]\label{prop: 4.2}
Consider a Borel action with discrete stabilizers of a locally compact
separable metrizable group $R$ on a standard Borel space $(Z, \ZZ)$. 

Then there is a Borel map $\sigma:  Z \to \MM_1(R)$ and a Borel
subset $E \subset Z$ such that $\lambda(Z \sm E) =0$ and such that,
for any discrete section $\Sigma \subset Z$ for the action of $R$, for
$\lambda_\Sigma$-a.e. $z_0 \in \Sigma$, for every $r \in R$ such that
$rz_0 \in E$, 
$$
\sigma(z_0) \simeq t_{r*} \sigma_\Sigma(rz_0).
$$
This map $\sigma$ is unique up to a set of $\lambda$-measure zero. 

For every $r \in R$ and every $z \in E$ such that $rz \in E$, we have 
\begin{equation}
\label{eq: 4.3}
\sigma(z) \simeq t_{r*}(\sigma(rz)).
\end{equation}

\end{prop}
The measure $\sigma(z)$ is called the conditional measure of $z$ along
the action of $R$. 

\begin{rem}
   Suppose $f: \R^{m+n} \rightarrow \R^{m+n}$ is a measurable map and suppose $\sigma \in \MM_1(\R^{m+n})$. Then, the notation $f_* \sigma \simeq \sigma$ means that there exists $c_f >0$ such that any $\widetilde \sigma \in \MM(\R^{m+n})$ whose image in $\MM_1(\R^{m+n})$ equals $\sigma$ satisfies $f_* \widetilde \sigma = c_f \sigma$. The notation $f_* \sigma = \sigma$ means that any $\widetilde \sigma \in \MM(\R^{m+n})$ whose image in $\MM_1(\R^{m+n})$ equals $\sigma$ satisfy $f_* \widetilde \sigma = \sigma$.
\end{rem}

In the following proposition from \cite{BQ1}, $\Gr(\R^d)$  denote the Grassmannian variety of $\R^d$. 

\begin{prop}[{\cite[Prop.~4.3]{BQ1}}]\label{prop: 4.3}
Let $(Z, \ZZ)$ be a standard Borel space endowed with a Borel action
of $\R^d$ with discrete stabilizers, and let $\lambda$ be a Borel
probability measure on $Z$. For $\lambda$-a.e. $z \in Z$, we denote by
$\sigma(z)$ the conditional measure of $z$ for the action of $\R^d$,
and 
$$
V_z = \{r \in \R^d: t_{r*}\sigma(z) = \sigma(z)\}_0,
$$
and by
$$
\lambda = \int_Z \lambda_z d\lambda(z)
$$
the distintegration of $\lambda$ along the map $Z \to \Gr(\R^d), \, z
\mapsto V_z$. Then for $\lambda$-a.e. $z \in Z$, the probability
measure $\lambda_z$ is $V_z$-invariant. 
\end{prop}

\section{Horocycle Flows}
\label{Horocycle Flows}

Following Benoist and Quint \cite{BQ1}, we define an action of $\R^{m+n}$ which plays a role analogous to that of the horocycle flow.  
\begin{defn}
    We define the horocycle flow as the action $\Phi$ of $\R^{m +n}$ given by, for any $v \in \R^{m +n}$ and $(b,k, x) \in B^{\tau, \X}$ as
    \begin{align}
        \label{def:action R}
\Phi_v(b,k,x)= (b, \left[I_{m+n},  a_{k} v \right]x).
    \end{align}
\end{defn}

We then have the following Lemma, see Lemma 6.10 in \cite{BQ1}. 

\begin{lem}
\label{lem: Equivalence of Action}
    For any $v \in \R^{m+n}$ and any $l \in \N$, we have for any $(c,x) \in B^{\tau, \X},$
    \begin{align}
    \label{eq:Action equivalence}
        T^{\tau,\X}_l \circ \Phi_v(c,x) = \Phi_{a_{-l} v} \circ T_l^{\tau,\X}(c,x).
    \end{align}
\end{lem}

\begin{proof}
Let $S$ denote the transformation of $B \times \R  \times X$
given by 
$$
S(b,k,x) = (Tb, k-\tau(b), b_0^{-1}x).
$$
So $S$ sends $B^{\tau, X}$ to points outside $B \times \R_+
 \times X$. We define the flow $\widetilde T_l^{\tau, X}$ on $B \times \R
 \times X$ by 
$$
\widetilde T_l^{\tau, X}(b,k,x) = (b, k+l, x).
$$
The flow $T_l^{\tau, X}$ is given, for $l \geq 0$ and $(b,k,x)
\in B^{\tau, X}$, by 
$$
T_l^{\tau, X} (b,k,x) = (S^p \circ \widetilde T_l^{\tau, X})(b,k,x)
$$
where $p \geq 0$ is the unique integer for which $ (S^p \circ \widetilde T_l^{\tau, X})(b,k,x) \in B^{\tau, X}$. Now define an action $\widetilde \Phi$ of $\R^{m+n}$ on $B
\times \R  \times X$ by:
$$
\widetilde \Phi_v(b,k,x) = (b,k, [I_{m+n},a_{k}(v)].x).
$$

Note that for $\beta$-a.e. $b \in B$ and every $x \in \X$, $k \in \R$ we have
\begin{align*}
    S\circ \widetilde \Phi_v (b, k, x) &= S( b, k, [I_{m+n}, a_{k} v] x) \\
    &= (T(b), k - \tau(b) , b_0^{-1}[I_{m+n}, a_{k} v] x ) \\
    &= (T(b), k- \tau(b), [I_{m+n}, a_{k- \tau(b)} v]  b_0^{-1}x ) \\
    &= \widetilde  \Phi_v (T(b), k, b_0^{-1}x) \\
    &= \widetilde  \Phi_v \circ S (b, k, x).
\end{align*}
Thus, 
\begin{align}
\label{eq: x1}
    S \circ \widetilde  \Phi_v = \widetilde  \Phi \circ S.
\end{align}
Also, for $\beta$-a.e. $b \in B$ and every $x \in \X$, $k \in \R$ we have
\begin{align*}
    \widetilde T_l^{\tau, \X} \circ \widetilde \Phi_v (b, k, x) &= \widetilde T_l^{\tau, \X}  ( b, k, [I_{m+n}, a_{k} v] x) \\
    &= (b, k +l  ,[I_{m+n}, a_{k} v] x ) \\
    &= \widetilde \Phi_{a_{-l}v}(b, k +l, x ) \\
    &= \widetilde  \Phi_{a_{-l}v} \circ \widetilde T_l^{\tau, \X}   (b,k,x).
\end{align*}
Thus, 
\begin{align}
\label{eq: x2}
    \widetilde T_l^{\tau, \X}  \circ \widetilde  \Phi_v = \widetilde  \Phi_{a_{-l}v} \circ \widetilde T_l^{\tau, \X} .
\end{align}
Equations \eqref{eq: x1} and \eqref{eq: x2} prove the lemma. 
\end{proof}

\begin{defn}
    Let us denote by $\sigma_0: B^{\tau, \X} \rightarrow \MM_1(\R^{m+n})$ the map given by `conditional measures of the probability measures of the probability measure $\beta^\X$ with respect to the horocyclic action of $\R^{m+n}.$' Also, denote by $t_v$ the translation of $\R^{m+n}$ by an element $v \in \R^{m+n}$.
\end{defn}

\begin{lem}
\label{lem: 6.11}
    There exists a measurable map $ \sigma: B^{\tau, \X} \rightarrow \MM(\R^{m+n})$ such that the composition of $\sigma$ with the natural projection map $\MM(\R^{m+n}) \rightarrow \MM_1(\R^{m+n})$ equals $\sigma_0$, and satisfy that for any $v \in \R^{m+n}$ and $(b,k, x) \in B^{\tau, \X}$, we have
    $$
    {t_v}_* \sigma(\Phi_v(b,k,x)) = \sigma(b,k,x).
    $$
\end{lem}
\begin{proof}
  Choose the discrete section of the action of $\R^{m+ n}$ as 
  $$
  \Sigma = \{(b, k,[A, 0]\Gamma): b \in B, 0 \leq k < \tau(b), A \in \SL_{m+n}(\R)\}.
  $$
  Then it is clear that if $z \in \Sigma$ and $v \in \R^{m+n}$ satisfy that $\Phi_v(z) \in \Sigma$, then $\Phi_v(z) =z$. Hence, for all $z \in \Sigma$ and $v \in \R^{m+n}$ such that $\Phi_v(z) \in \Sigma$, we have
  \begin{align}
      \label{eq: w1}
      \sigma_\Sigma(z) = t_{v*} \sigma(\Phi_v(z)).
  \end{align}
  Hence, we can define for all $[A, v]\Gamma \in \X$
  $$
  \sigma((b, k,[A, v]\Gamma)) = (t_{-v})_* \sigma_\Sigma(b, k,[A, 0]\Gamma),
  $$
  which is well defined by \eqref{eq: w1}. This map $\sigma$ satisfies the conditions of lemma.
\end{proof}

\begin{lem}
\label{lem: 6.12}
    For any $l \geq 0$, for $\beta^{\tau, \X}$-a.e.  $(c, x) \in B^{\tau, \X}$, we have
    $$
    \sigma(T_l^{\tau,\X}(c,x)) = (a_{-l})_* \sigma(c,x)
    $$
\end{lem}
\begin{proof}
First of all, note that for $\beta^{\tau, \X}$-a.e.  $(c, x) \in B^{\tau, \X}$, we have
$$
    \sigma(T_l^{\tau,\X}(c,x)) \simeq (a_{-l})_* \sigma(c,x)
$$
This is a result of the uniqueness  of $\sigma_0$, equality \eqref{eq:Action equivalence} and the fact that for $\beta$-a.e. $b \in B$, for any $p \in
\N$, the action of $b_{p}^{-1} \cdots b_1^{-1}$ induces an
isomorphism between the measure spaces $(\X, \nu_b)$ and $(\X,
\nu_{T^pb})$. To show that equality hols, suppose that $x= [A,v]\Gamma$. 

Then, note that both $\sigma(T_l^{\tau,\X}(c,x))$ and $(a_{-l})_* \sigma(c,x)$ both give measure one to the set $a_{-l}A[0,1]^{m+n}$.
\end{proof}

\begin{cor}\label{cor: 6.13}
The map $\sigma: B^{\tau, \X}
\to \MM(\R^{m+n})$ is $\QQ^{\tau, \X}_\infty$-measurable. 
\end{cor}
\begin{proof}
It suffices to show that for any $l \geq 0$, it is $\QQ^{\tau, \X}_l$-measurable. This results from the equality, for
$\beta^{\X}$-a.e. $(c,x) \in B^{\tau, \X}$, $\sigma(c,x) =
(a^{l})_*(\sigma(T^{\tau,\X}_l(c,x))).$ 
\end{proof}

\section{The exponential drift}
\label{The exponential drift}
The main goal of this section is to study ``exponential drift". This is encoded in the following key proposition:

\begin{prop}\label{prop: 7.1}
For $\beta^{\tau, \X}$-a.e. $(c,x) \in B^{\tau, \X}$, for any $\e>0$, there exists a nonzero element $v \in \R^{m+n}$ of norm at most $\e$ and an element $(c',x')$ such that 
\begin{equation}
  \label{eq: 7.1}
\sigma(c,x) = \sigma(c',x')= \sigma(\Phi_v(c',x')). 
\end{equation}

\end{prop} 

\begin{proof}
The argument begins in the same fashion as that of Benoist and Quint. Namely, we may assume that $\MM(\R^{m+n})$ is endowed with the topology of a complete separable metric space. We may also endow $B^\tau$ with the topology of a compact
metric space and endow $B^\tau \times X$ with the product topology of this topology and the usual topology on $X$. 

Let $\alpha>0$; by Lusin's theorem, there is  a compact subset $K$ in $B^{\X}$ such that $\beta^{\X}(K^c) < \alpha^2$ and such that the function $\sigma$ is uniformly continuous on $K$. Due to Lemma \ref{lem: 6.11}, we can assume that $K$ is invariant under the horocycle flow. 

Since $\sigma$ is $\QQ_\infty^{\tau,\X}$-measurable, it is
$\QQ_l^{\tau ,X}$-measurable for each $l \geq 0$, and hence, by Proposition \ref{cor:BQ 3.8}, we can assume that for every $(c,x)
\in K$, for $\beta$-a.e. $a \in B$, for any $l \geq 0$ rational, we
have $\sigma(h_{l, c,x}(a)) = \sigma(c,x)$.\\

We will study the function $\EE \left(1_K |\QQ_\infty^{\X} \right)$. On the one hand, it is bounded above by 1 and on the other,

\begin{align}
    \label{eq: 7.2}
\int_{B^{\X}} \EE\left(1_K | \QQ_\infty^{\X}\right)(c,x) d\beta^{\X}(c,x) = \beta^{\X}(K) > 1-\alpha^2.
\end{align}
We therefore have that $\EE \left(1_K | \QQ_\infty^{\X}\right)$ is bounded below
by $1-\alpha$ on a set of measure  $1-\alpha$. There is then a
compact subset $L_1 \subset E$ in $B^{\X}$ such that $\beta^{\X}(L_1^c) < \alpha$ and such that, for every $(c,x) \in L_1$, we have
\begin{align}
    \label{eq: 7.3}
\EE \left(1_K | \QQ_\infty^{\X} \right)(c,x) > 1-\alpha.
\end{align}

Now by the Martingale convergence theorem, for $\beta^{\X}$-a.e. $(c,x) \in B^{\X}$, we have 
\begin{align}
    \label{eq: 7.4}
\lim_{l \to \infty} \EE\left(1_K | \QQ_l^{\X} \right)(c,x) = \EE \left(1_K | \QQ_\infty^{\X} \right)(c,x).
\end{align}

By Corollary \ref{cor:BQ 3.8}, we may also suppose that for every $(c,x) \in L_1$ and $l \in \Q$, the left-hand side of \eqref{eq: 7.4} is given by the formula \eqref{eq: BQ 3.11}. Thus
\begin{align}
  \label{eq: 7.5}
\EE \left(1_K | \QQ_l^{\X} \right)(c,x) = \int_B 1_K (h_{l, c, x}(a)) d\beta(a),  
\end{align}
where 
$$
h_{l, c,x}(a) = (c',x') \text{ with } c' = h_{l, c}(a) \text{
  and } x' = \rho_l(c')^{-1} \rho_l(c)x.
$$

We now appeal to Egorov's theorem which guarantees us that outside a subset of $L_1$ of arbitrarily
small $\beta^{\X}$-measure, the convergence in \eqref{eq: 7.4} above is
uniform on $L_1$. Therefore, after removing a subset of $L_1$ of small
measure, there exists $l_0 \geq 0$ such that for 
every rational $l \geq l_0$, for every $(c,x) \in L$, we have 
\begin{align}
  \label{eq: 7.6}{
\EE\left(1_K | \QQ_l^{\X} \right)(c,x) \geq 1-\alpha.
}  
\end{align}

Let $L= (\bigcup_{v \in \R^{m+n}} \Phi_v(L_1))\cap K$. Clearly, we have that $L$ is compact and since $L \subset K$, we have that $\sigma$ is uniformly continuous on $L$. Also, we have that for every rational $l \geq l_0$ and $(c,x) \in L$, we have
$$
    \int_B 1_K (h_{l, c, x}(a)) d\beta(a) \geq 1-\alpha.
$$

Note that the $\beta^{\tau,\X}$-measure of $L^c$ is at most $2\alpha$ and since 
$\alpha$ was chosen arbitrarily small, it suffices to prove \eqref{eq: 7.1}
 for $\beta^{\X}$-a.e. $(c,x) \in L$.\\ 

For $c \in B^\tau$, define $L_c =\{x \in \X: (c,x) \in L\}$. We will show that for $\beta^\tau$-a.e. $c$, we have that for $\mu_c$-a.e. $x \in L_c$ satisfies \eqref{eq: 7.1}. For the rest of the proof, we fix $c \in B^\tau$ such that $\mu_{c}(L_c) >0$.

Note that due to invariance under the horocycle flow of the set $L$, we get that 

$$L_c = \{[A,v]\Gamma: A\Gamma' \in \pi(L_c), v \in \R^{m+n}\}.$$ 

We will denote by $L'_c$ the set $\pi(L_c)$.

Note that $\mu_c(L_c)>0$ implies that using Lemma \ref{lem:Imp Project mu b equal mu}, we may assume that $\mu_{\X'}(L_c')>0$. 

Set $D = (m+n)^2 -1$. Using the manifold structure of $\X'$, we can cover $L_c'$  by finitely many open sets $\{U_i: i \in I\}$ in $\X'$ and can fix embeddings $\{f_i: U_i \rightarrow \R^{D}\}_{i \in I}$. Clearly for all $i \in I$, the pushforward of $\mu_{\X'}|_{U_i}$ under $f_i$ is absolutely continuous with respect to Lebesgue measure on $\R^{D}$, hence we get that $f_i(U_i \cap L_c')$ is a set of positive Lebesgue measure in $\R^{D}$. 

Fix $i \in I$ and $A\Gamma' \in U_i$ such that 
\begin{align}
    \label{eq: Lebesgue Density theorem}
    \frac{\mu_{\R^D}(B_\delta(f_i(A\Gamma') ) \cap f_i(U_i \cap L_c') )}{\mu_{\R^D}(B_\delta(f_i(A\Gamma')))} \rightarrow 1
\end{align}
as $\delta \rightarrow 0$, where $\mu_{\R^D}$ denotes the Lebesgue measure on $\R^D$ and $B_\delta(x) \subset \R^D$ denotes the open ball around $x$ of radius $\delta$. Note that by the Lebesgue density theorem, \eqref{eq: Lebesgue Density theorem} holds for $\mu_{\X'}$- a.e. point in $U_i \cap L_c'$. Let $L_c''$ consist of all points in $L_c'$ such that \eqref{eq: Lebesgue Density theorem} holds for some $i \in I$. Clearly, $\mu_{\X'}(L_c' \setminus L_c'') = 0$ and the definition of $L_c''$ is independent of the open cover $\{U_i\}_i$ and the embedding $\{f_i\}_i$.

For the rest of the proof, fix $A\Gamma' \in L_c''$. Also fix $v \in \R^{m+n}$ to be arbitrary and define $x = [A,v]\Gamma$. We will prove \eqref{eq: 7.1} for $(c, x)$. This will prove the Proposition.

Fix $\{E_{ij}: i \neq j\} \cup \{E_{ii}- E_{m+n, m+n}: 1 \leq i < m+n\}$ as a basis for $\sL_{m+n}(\R)$. Let $B_\gamma(\sL_{m+n}(\R))$ denote the Euclidean ball in $\sL_{m+n}(\R)$ with respect to the above basis. 
Fix $\gamma>0$ small enough so that the map $f: B_\gamma (\sL_{m+n}(\R)) \rightarrow \X$ as $ f(M ) = \exp(M).x$ is a diffeomorphism onto its image. Then, by \eqref{eq: Lebesgue Density theorem}, we get that
\begin{align}
    \label{eq: Lebesgue useful}
    \frac{\mu_{\R^D}(B_\delta(\sL_{m+n}(\R)) \cap f^{-1}(L_c) )}{\mu_{\R^D}(B_\delta(\sL_{m+n}(\R))} \rightarrow 1
\end{align}
as $\delta \rightarrow 0$.

Fix $\delta>0$ small enough, to be chosen later. Let $\kappa, T, C_1, C_2, C_3$ be as in Lemma \ref{lem: Condition on Fractal}. The notation $Y_l(c), A_{l,c,a}, v_{l,c,a}$ for $l \geq T$, will be as in Lemma \ref{lem: Condition on Fractal}.

The goal now is to find a sequence $(M_j)_j \in B_\gamma(\sL_{m+n, m+n}(\R))$, $(l_j)_j \in \R_+$ and $(a_j)_j \in B$ such that
\begin{itemize}
    \item $M_j \rightarrow 0$ as $j \rightarrow \infty$.
    \item $l_j \rightarrow \infty$ as $j \rightarrow \infty$.
    \item $f(M_j) \in L_c$ for all $j$.
    \item $h_{l_j, c,x}(a_j), h_{l_j, c, f(M_j)}(a_j) \in K$ for all $j$.
    \item For all $j$, $h_{l_j,c,f(M_j)}(a_j) = \exp([M_j', v_j]) h_{l_j,c,x}(a_j)$ such that $M_j' \rightarrow 0$ as $j \rightarrow  \infty$ and $ \delta \leq \|v_j\| \leq \e$.
    \item $\sigma(h_{l_j,c,x}(a))= \sigma(c,x)$ for all $j$.
    \item $\sigma(h_{l_j, c, f(M_j)}(a_j)) = \sigma(c, \exp({M_j})x)$ for all $j$.
\end{itemize}
If we can find such a sequence, then up to passing to a subsequence, we will have that
\begin{itemize}
    \item The sequence $h_{l_j, c,x}(a_j)$ has a limit $(c',x') \in K$.
    \item The sequence $v_j$ has a limit $v$ such that $0< \delta \leq \|v\| \leq \e$.
    \item The sequence $h_{l_j, c, f(M_j)}(a_j)$ has a limit and equals $$\exp([0_{m+n}, v]) . (c',x') = [I_{m+n},v].(c',x') = \Phi_v(c',x').$$
\end{itemize}

Then, since all the limits considered have values in $K$ and $L$, we have
\begin{align*}
    \sigma(c',x') = \lim_{j \rightarrow \infty} \sigma(h_{l_j, c,x}(a_j)) = \lim_{j \rightarrow \infty} \sigma(c,x) = \sigma (c,x), \\
    \sigma(\Phi_v(c',x')) = \lim_{j \rightarrow \infty} \sigma(h_{l_j, c, f(M_j)}(a_j)) = \lim_{j \rightarrow \infty} \sigma(c, \exp({M_j})x) = \sigma(c,x).
\end{align*}
Thus, \eqref{eq: 7.1} will hold for $(c,x)$. Hence, it is enough to prove the existence of a sequence satisfying the above inequalities.

Let $\e' = \e/ e^{C_1} C_2$ and let $l \geq \max\{T, l_0\}$ be a rational number.
Consider the set $X_l \subset B_{\e'}(\sL_{m+n}(\R)) \times B $ defined as the set of all points $(M,a)$ such that
\begin{itemize}
    \item $a \in Y_l(c)$, 
    \item $f(e^{-l}M) \in L_c$,
    \item $h_{l,c,x}(a) \in K$,
    \item $h_{l,c,f(e^{-l}M)}(a) \in K$,
    \item $\sigma(h_{l,c,x}(a))= \sigma(c,x)$,
    \item $\sigma(h_{l,c,f(e^{-l}M)}(a)) = \sigma(c,f(e^{-l}M))$,
    \item $\|M.v_{l,c,a}\|> \delta$.
\end{itemize}
The measure of this set with respect to the product of Lebesgue measure $\mu_{\sL_{m+n}(\R)}$ on $\sL_{m+n}(\R)$ and $\beta$ is greater than equal to the measure of the set $Z_1 \setminus (Z_2 \cup Z_3 \cup Z_4 \cup Z_5 \cup Z_6)$, where
\begin{itemize}
    \item $Z_1 = \{(M,a): a \in Y_l(c) \}$,
    \item $Z_2 = \{(M,a): f(e^{-l}M) \notin L_c \}$,
    \item $Z_3 = \{(M,a): h_{l,c,x}(a) \notin K\}$,
    \item $Z_4 = \{(M,a): f(e^{-l}M) \in L_c, h_{l,c,f(e^{-l}M)}(a) \notin K\}$,
    \item $Z_5 = \{(M,a): a \in Y_l,   \|e^{-l}M.v_{l,c,a}\| < \delta \}$,
    \item $Z_6= \{(M,a):f(e^{-l}M), \sigma(h_{l,c,x}(a))\neq \sigma(c,x), \sigma(h_{l,c,f(e^{-l}M)}(a)) \neq \sigma(c,f(e^{-l}M)) \}$.
\end{itemize}
Clearly, 
\begin{align*}
    \mu_{\sL_{m+n}(\R)} \otimes \beta (Z_1) &\geq \mu_{\sL_{m+n}(\R)}(B_{\e'}(\sL_{m+n}(\R))). \kappa, \\
    \mu_{\sL_{m+n}(\R)} \otimes \beta (Z_2) &\rightarrow 0 \text{ as } l \rightarrow \infty, \\
    \mu_{\sL_{m+n}(\R)} \otimes \beta (Z_3) &< \mu_{\sL_{m+n}(\R)}(B_{\e'}(\sL_{m+n}(\R))) .\alpha, \\
    \mu_{\sL_{m+n}(\R)} \otimes \beta (Z_4) &< \mu_{\sL_{m+n}(\R)}(B_{\e'}(\sL_{m+n}(\R))). \alpha, \\
    \mu_{\sL_{m+n}(\R)} \otimes \beta (Z_6) &=0.
\end{align*}

We will now estimate the measure of $Z_5$. Note that if $a \in Y_l(c)$, we have $e^{-C_1}C_2 < e^{-l}\|v_{l,c,a}\|_{2}$. Then the measure of the set $$\{M \in B_{\e'}(\text{Mat}_{m \times n}(\R)): \|e^{-l}M.v_{l,c,a}\| <  \delta \}$$  equals the measure of the set $$\{M \in B_{\e'}(\text{Mat}_{m \times n}(\R)): \|e^{-l}M.e_1\| <  \delta \},$$ which is less than $(\e')^{(m-1)n}. \delta^n$ (up to a constant). 
Thus, we have
\begin{align*}
    &\mu_{\sL_{m+n}(\R)} \left(\left\{M = \begin{pmatrix}
    M_{11} & M_{12} \\ M_{21} & M_{22} 
\end{pmatrix} \in B_{\e'}(\sL_{m+n}(\R)): \|e^{-l}M.v_{l,c,a}\| < \delta \right\} \right) \\
&\ll (\e')^{mn + m^2 + n^2 -1}. \mu_{\text{M}_{m \times n}(\R)}\left(M_{21} \in B_{\e'}: \|e^{-l}M_{21}.e_1\| <  \delta \right)  \\
& \ll (\e')^{mn + m^2 + n^2 -1 +(m-1)n}. \delta^n.
\end{align*}

Thus by Fubini's Theorem, the measure of $Z_5$ is less than $(\e')^{mn + m^2 + n^2 -1 +(m-1)n}. \delta^n$.

If $\delta>0$ is small enough and $\alpha < \kappa/4$ and $l $ is large enough, we get that
\begin{align*}
    &\mu_{\sL_{m+n}(\R)} \otimes \beta(X_l) \\
    &\geq \mu_{\sL_{m+n}(\R)} \otimes \beta(Z_1) - \sum_{i=2}^6 \mu_{\sL_{m+n}(\R)} \otimes \beta(Z_i) \\
    & > 0.
\end{align*}
Hence, along a sequence $l_j \in \Q$ such that $l_j \rightarrow \infty$, one can always choose $(M_j'', a_j) \in X_{l_j}$. Then with $M_j = e^{-l_j} M_j''$, the required properties are satisfied. 

\end{proof}

\section{Disintegration of $\nu_b$ along the stabilizers}
\label{Disintegration of nu along stabilizers}

Following Benoist and Quint, we consider the connected stabilizers of the measures $\sigma_0(c,x)$, namely the subspaces of $\R^{m+n}$ defined by
$$
J(c,x) = \{ v\in \R^{m+n} : t_{v*}\sigma_0(c,x) = \sigma_0(c,x) \}_0.
$$

As in \cite{BQ1}, we then have

\begin{prop}\label{prop: 7.4}
For $\beta^{\X}$-a.e. $(c,x) \in B^{\tau,\X}$, we have 
$$J(c,x) \neq \{0\}.$$
\end{prop}
\begin{proof}
We will show that for $\beta^{\X}$-a.e. $(c,x)$ and every
$\e>0$, the stabilizer of $\sigma(c,x)$
contains a nonzero vector of norm at most $\e$. 

The drift argument (Proposition \ref{prop:
  7.1}) gives us that for $\beta^{\X}$-a.e. $(c,x) \in B^{\tau,\X}$ and every
$\e>0$, there exists a nonzero vector $v \in \R^{m+n}$ of norm at most $\e$ and an
element $(c',x')$ such that 
$$
\sigma(\Phi_v(c',x')) = \sigma(c',x') = \sigma(c,x). 
$$
By applying Lemma \ref{lem: 6.11} to this element $(c',x')$, we find 
$$
t_{v*}\sigma(\Phi_v(c',x')) = \sigma(c',x')
$$
and hence 
$$
t_{v*}\sigma(c,x) = \sigma(c,x).
$$
which gives that $t_{v*}\sigma_0(c,x) = \sigma_0(c,x).$
The vector $v$ is indeed in the stabilizer of $\sigma(c,x)$. The stabilizer is non-discrete and closed. It thus contains a nonzero linear subspace of $\R^{m+n}$. 
\end{proof}

For $\beta$-a.e. $b \in B$, and $\mu_b$-a.e. $x \in X$, we denote by
$\sigma_{b,x} \in \MM(\R^{m+n})$ the conditional measure at $x$ of $\mu_b$
for the action on $X$ of $\R^{m+n}$ as a subgroup of $G$, and we denote $V_{b,x} \subset \R^{m+n}$ the
connected component of the stabilizer of $\sigma_{b,x}$ in $\R^{m+n}$. We then have the analogue of Proposition 7.5 in \cite{BQ1}.

\begin{prop}\label{prop: 7.5}
For $\beta^X$-a.e. $(b,x) \in B^X$, we have $\sigma_{b,x} \simeq \pi_G(b_{0})_* \sigma_{T^X(b,x)}, \,
V_{b,x} = \pi_G(b_0)\left(V_{T^X(b,x)} \right)$ and $V_{b,x} \neq 0.$ 
\end{prop}

\begin{proof}
For $\beta$-a.e. $b
\in B$, we have $\mu_{Tb} = (b^{-1}_0)_* \mu_b$ and, for every $x \in X$ and
$v \in \R^{m+n}$, 
$$
T^X(b, [I_{m+n},v]x) = (Tb, [I_{m+n},\pi_G(b^{-1}_0) v]b^{-1}_0x).
$$
This gives the first equality and the second equality is immediate. The fact that $V_{b,x}$ is nonzero follows from Proposition \ref{prop:
  7.4} and the equality, for $\beta^{\X}$-a.e. $(b,x) \in
B^{\X}$, $V_{b,x} = a_{k}.(J(b,k,x))$. 
\end{proof}

The disintegration of $\beta^X$ along the map $(b,x) \mapsto (b,
V_{b,x})$ can be written as 
$$
\mu_b = \int_X \mu_{b,x} d\nu_b(x)
$$
where, for $\beta^X$-a.e. $(b,x) \in B^X$, the probability measure
$\mu_{b,x}$ on $X$ is supported on the fiber $\{x' \in X: V_{b,x'} =
V_{b,x} \}$. 

\begin{prop}\label{prop: 7.6}
For $\beta^{\tau,X}$-a.e. $(b,x) \in B^X$, the probability measure
$\nu_{b,x}$ is $V_{b,x}$-invariant and has the equivariance property
$\nu_{b,x} = b_{0*} \nu_{Tb, b^{-1}_0x}. $ 
\end{prop}
\begin{proof}
The first assertion follows from Proposition \ref{prop: 4.3}. 

The second assertion follows from the equality $\mu_b =
b_{0*}\mu_{Tb}$, from Proposition \ref{prop: 7.5}, and from the
disintegration of measures. 
\end{proof}

\section{Proof of Theorem \ref{Main Random walk}}
\label{Proof of Theorem Main Random walk}
This section aims to prove Theorem \ref{Main Random walk}.

\begin{defn}
    We will call a probability measure $\alpha$ on $\T^{m+n}= \R^{m+n}/\Z^{m+n}$ \emph{homogeneous} if there is a connected closed subgroup $H$ of $\T^{m+n}$ containing $\supp(\alpha)$ such that $\alpha$ equals Haar measure on $H$.
\end{defn}
\begin{defn}
    A measure $\alpha$ will be called a \emph{translated homogeneous toral measure} if $\alpha= g_* \alpha'$ for some $g \in G$, where $\alpha'$ equals the measure obtained by the pushforward of a homogeneous measure $\rho''$ on $\T^{m+n}$, under map $w \mapsto [I_{m+n},w]\Gamma$.
\end{defn}

\begin{prop}
\label{Prop: Decomposition implies Haar}
    Suppose $\mu$ is a measure on $\X$ which is $\nu$-ergodic and satisfies $(\pi_{\X})_* \mu_{\X} = \mu_{\X'}$. Assume that there exists a probability measure $\Theta$ on $\Prob(\X)$, the set of Radon probability measures on $\X$ such that
    \begin{itemize}
        \item $\mu= \int_{\Prob(\X)} \rho \, d\Theta(\rho)$.\\
        \item For $\Theta$-a.e. $\rho$ is a translated homogenous toral measure.\\
        \item The measure $\Theta$ is $\nu$-stationary.
    \end{itemize}
    Then, $\mu= \mu_{\X}$.
\end{prop}
\begin{proof}
We can assume that $\Theta$ is $\nu$-ergodic, since $\mu$ is $\nu$-ergodic. Consider the set $M(\X) \subset \Prob(\X)$ of all translated homogeneous toral measures on $\X$ and the set $G \setminus M(\X)$. Since the subgroups of $\T^{m+n}$ are countable, so is $G \setminus M(\X)$. The pushforward of $\Theta$ under the natural map $M(\X) \rightarrow G \setminus M(\X)$, say $\Bar{\Theta}$, is a $\nu$-stationary ergodic probability measure on a countable set. Since the action of $G$ on $G \setminus M(\X)$ is trivial, this implies that $\Bar{\Theta}$ is supported on a single point (see \cite[Lem.~8.3]{BQ1}).

    Thus, $\Theta$ is supported on a unique orbit $G \alpha$, where $\alpha$ can be assumed to be the pushforward of a homogeneous measure $\alpha'$ on $\T^{m+n}$ under the map $w \mapsto [I_{m+n},w]\Gamma$ and the pre-image of the subgroup $\text{supp}(\alpha')$ under the natural map $\R^{m+n} \rightarrow \R^{m+n}/ \Z^{m+n}$ equals $V_{\alpha}$, which is a rational subspace of $\R^{m+n}$.

    Clearly, $G \alpha \simeq G/G_{\alpha}$ and we have
    \begin{align*}
        G_\alpha &= \{g \in G: g_* \alpha = \alpha\} \\
        &= \{[A, w] \in \SL_{m+n}(\Z) \times \R^{m+n}: A.V_\alpha  = V_\alpha, w \in V_\alpha \}.
    \end{align*} 

    Now, we have two cases to consider:\\ 
    
    {\bf Case 1: $V_\alpha= \R^{m+n}$.} Then, the measure $\alpha'$ equals Haar measure on $\T^{m+n}$. Using the fact that $\mu= \int_{G \alpha} \rho \, d\Theta(\rho)$, this implies that $\mu$ is invariant under the action of $\R^{m+n}$. Since $\pi_*(\mu) = \mu_{\X'}$, this implies that $\mu = \mu_{\X}$.\\

    {\bf Case 2: $V_{\alpha} \subsetneq \R^{m+n}$.} Then we have that $G_\alpha$ is a discrete subgroup which is not a lattice in $G$. Let $Z \subset G$ denote the set of all $[A,w] \in G$ such that for $\nu^{\otimes \N} $-a.e. $(b_j)_j \in B$, we have that the sequence $(b_j \cdots b_1 [A,w] G_{\alpha})_j$ diverges to infinity in $G/ G_\alpha$. Then we have that  $m_G(Z^c)=0$. Clearly, $Z$ is invariant under the left action of $\R^{m+n}$. By removing a set of measure zero, we may assume that $Z$ is invariant under the right action of $\Gamma$. Then, $Z$ is right invariant under the action of $G_\alpha$ as well.

    Note that the set $Z.G_\alpha$ has measure zero with respect to $\Theta$, since for $\Theta$-a.e. $x \in G/G_\alpha \simeq G \alpha$, the random walk equidistributes with respect to $\Theta$. Thus,
    \begin{align*}
        1 &= \mu_{\X'}(\pi_G(Z.\Gamma)) \\
        &= \mu(Z.\Gamma) \\
        &= \int_{G.\alpha} \rho(Z.\Gamma) \, d\Theta(\rho) \\
        &= \int_{G/G_\alpha} ([A,w]_* \alpha)(Z.\Gamma) \, d\Theta([A,w]G_\alpha) \\
        &= \int_{G/G_\alpha} {1}_{Z.\Gamma}([A,w]\Gamma) \, d\Theta([A,w]G_\alpha) \\
        &= \int_{G/G_\alpha} {1}_{Z.G_\alpha}([A,w]G_{\alpha}) \, d\Theta([A,w]G_\alpha) \\
        &= 0.
    \end{align*}
    This means that Case 2 is not possible and proves the Proposition.
\end{proof}

\begin{proof}[Proof of Theorem \ref{Main Random walk}]
    We may assume that $\mu$ is a $\nu$-ergodic stationary measure. For $\beta^X$-a.e. $(b,x) \in B^X$, the decomposition of $\mu_{b,x}$
into $V_{b,x}$-ergodic components can be written simultaneously
in the form 
\begin{equation}
    \label{eq:8.1}
    \mu_{b,x} = \int_X \zeta(b,x') d\mu_{b,x}(x'),
\end{equation}
where $\zeta: B^X \to \Prob(\X)$ is a $\cB^X$-measurable map such that, for
$\beta^X$-a.e. $(b,x) \in B^X$, the restriction of $\zeta$ to the
fiber $\{(b,x'): V_{b,x'} = V_{b,x}\}$ is constant along
$V_{b,x}$-orbits. 

It is easy to see that for $\beta^\X$-a.e. $(b, [A,v]\Gamma)$, we have that $V_{b,[A,v]\Gamma}$ must be of the form $A.V'_{b,[A,v]\Gamma}$, where $V'_{b,[A,v]\Gamma}$ is a rational subspace of $\R^{m+n}$. Hence, for $\beta^\X$-a.e. $(b, x)$, we have that $\zeta(b,x)$ is of the form $([A_{\rho},v_\rho])_{*} \rho'$, where $\rho'$ is a measure on $\X$, obtained by the pushforward of a homogeneous measure $\rho''$ on $\R^{m+n}/ \Z^{m+n}$, under map $w \mapsto [I_{m+n},w]\Gamma$. 

The uniqueness of the ergodic decomposition, and Propositions
\ref{prop: 7.5} and \ref{prop: 7.6}, prove that, for
$\beta^X$-a.e. $(b,x) \in B^X$, we have 
\begin{align}
  \label{eq: 8.2}  
  \zeta(b,x) = (b_0)_* \zeta(T^X(b, x)). 
\end{align}
This implies that the image probability measure $\Theta =
\zeta_* \beta^X$ is therefore a $\mu$-stationary probability measure
on $\Prob(\X)$ (see e.g. \cite[Lem. 3.2(e)]{BQ1}). 

Also, note that
\begin{align*}
    \mu &= \int_{B} \mu_{b} \, d\beta(b) \\
    &=  \int_{B} \int_{\X} \mu_{b,x} \, d\mu_b(x)  d\beta(b) \\
    &=  \int_{B} \int_{\X} \int_{\X} \zeta(b,x')\, d \mu_{b,x}(x') d\mu_b(x)  d\beta(b) \\
    &= \int_{B} \int_{\X} \zeta(b,x')\,  d\mu_b(x')  d\beta(b) \\
    &= \int_{\Prob(\X)} \rho \, d\Theta(\rho).
\end{align*}

Thus, $\mu$ satisfies the conditions of Proposition \ref{Prop: Decomposition implies Haar}. Hence by Proposition \ref{Prop: Decomposition implies Haar} we get that $\mu= \mu_{\X}$. This proves the Theorem.
\end{proof}

\section{Random Genericity: Proof of Theorem \ref{Random Genericity}}
\label{Proof of Theorem Random Genericity}
\begin{prop}\label{prop: Sparse Equidistribution}
    Fix $(t_i)_i$ a sequence in the closed interval $[c,d]$, for some $0<c<d<\infty$.
    Fix $x_0 \in \X'$ which is Birkhoff generic. Suppose that the measure $\mu$ is a weak limit of the sequence of probability measures 
    \begin{align}
    \label{eq:r3}
        \frac{1}{p} \sum_{i=1}^p \delta_{a_{t_1 + \cdots + t_i}x_0},
    \end{align}
    along a subsequence. Then $\mu$ is a probability measure and we have 
    $$
    \int_{[0,c]} (a_{-l})_*\mu \, dl \ll \mu_{\X'}  \ll \int_{[0,d]} (a_{-l})_*\mu \, dl,
    $$
    where $\alpha \ll \beta$ means that the measure $\alpha$ is absolutely continuous w.r.t $\beta$.
\end{prop}
\begin{proof}
Let $\mu$ be a weak limit of the left-hand side of \eqref{eq:r3} along a subsequence, say $(p_k)_k$.

    First of all, let us show that $\mu$ is a probability measure. Let $\e>0$ be arbitrary. We will show that there is a compact subset $K_\e \subset\X'$ such that $\mu(K_\e)> 1- \e$. This will prove that $\mu$ is a probability measure. To see this, note that since $x_0$ is Birkhoff generic, there exists a compact subset $K_\e' \subset \X'$ and $T_\e$ such that for all $T> T_\e$ we have
    \begin{align}
        \label{eq:r1}
        \frac{1}{T} \int_{0}^T \delta_{a_t x}((K_\e')^c) < \frac{c}{d}\e.
    \end{align}
    Define $K_\e= \bigcup_{l\in [0,d]} a_{l} K_\e'$. Then for $p_k> T_e/c$ we have $t_1 + \cdots + t_{p_k} > T_\e$, and 
    \begin{align*}
        \frac{1}{p_k} \sum_{i=1}^{p_k} \delta_{a_{t_1 + \cdots + t_i}x_0}(K_\e^c) &\leq \frac{d}{t_1 + \cdots + t_{p_k}} \sum_{i=1}^{p_k} \left(\frac{1}{t_i} \int_{t_1+ \cdots + t_{i-1}}^{t_1 + \cdots + t_i} \delta_{a_{l} x}((K_\e')^c)  \right) \\
        &\leq \frac{c}{d} \frac{1}{t_1 + \cdots + t_{p_k}} \int_{0}^{t_1 + \cdots + t_{p_k}} \delta_{a_{l} x}(K_\e) \\
        &\leq \e.
    \end{align*}
    This proves that $\mu$ is a probability measure.

    Now, by further passing to a subsequence of $(p_k)_k$, we may assume the existence of a weak limit of the sequence of the measures on $\X' \times [c,d]$ 
    $$
     \frac{1}{p_k}\sum_{i=1}^{p_k} \delta_{a_{t_1+ \cdots + t_i}x_0,t_i},
    $$
    say, $\Tilde{\mu}$. 
    Since the pushforward of $\Tilde{\mu}$ along the natural projection map $\X' \times [a,b] \rightarrow \X'$ must equal $\mu$, we get that $\Tilde{\mu}$ is a probability measure.

    Fix a function $f \in \C_c(\X')$. Define $F: \X' \times \R \rightarrow \Complex$ as
    $$
    F(x,t)= \int_{-t}^{0} f(a_l x) \, dl.
    $$
    Note that $|F(x,t)| \leq \|f\|_{\infty} .d$, hence $F$ is a continuous bounded function on $\X' \times [c,d]$. Also, define $\phi:\X' \times [c,d] \rightarrow \R $ as $\phi(x,l)=l$. Then, we have
    \begin{align*}
         \int_{\X' \times \R} \int_{-t}^0 f(a_lx)\, dl d\Tilde{\mu}(x,t) &=\int_{\X' \times \R} F(x,t)\, d\Tilde{\mu}(x,t) \\
        &= \lim_{k \rightarrow \infty} \frac{1}{p_k }\sum_{i=1}^{p_k}  F(a_{t_1+ \cdots + t_i}x_0,t_i) \\
        &= \lim_{k  \rightarrow \infty} \frac{1}{p_k }\sum_{i=1}^{p_k}  \int_{-t_i}^{0} f(a_{l+ t_1 + \cdots + t_i}x_0) \, dl \\
        &= \lim_{k  \rightarrow \infty} \frac{1}{p_k } \int_{0}^{t_1 + \cdots + t_{p_k} } f(a_{l}x_0) \, dl \\
        &= \lim_{k  \rightarrow \infty} \left(\frac{1}{t_1+ \cdots+ t_{p_k}} \int_{0}^{t_1 + \cdots + t_{p_k}}  f(a_{l}x_0) \, dl \right) \left( \frac{t_1 + \cdots + t_{p_k}}{p_k} \right) \\
        &=\left(\int_{\X'} f(x)\, d\mu_{\X'}(x) \right). \lim_{k \rightarrow \infty} \left( \frac{1}{p_k}\sum_{i=1}^{p_k} \phi(a_{t_1+ \cdots + t_i}x_0,t_i) \right)\\
        &= \left(\int_{\X'} f(x)\, d\mu_{\X'}(x) \right) \left( \int_{\X' \times [c,d]} \phi(x,l) \, d\Tilde{\mu}(l) \right),
    \end{align*}
    where in the last equality we have used the fact that $x_0$ is Birkhoff generic. Since, this holds for all $f \in \C_c(\X')$, we get that
    \begin{align*}
        \left( \int_{\R} t \, d\mu_{\R}(t) \right). \mu_{\X'} &= \int_{\X' \times [c,d]} \int_{-t}^{0} \delta_{a_l x} \, dl d\Tilde{\mu}(x,t)\\
        &\ll \int_{\X' \times [c,d]} \int_{-b}^0 \delta_{a_l x}  \, dl d\Tilde{\mu}(x,t)\\
        &= \int_{0}^d (a_{-l})_*\left( \int_{\X' \times \R}  \delta_{x}  \, d\Tilde{\mu}(x,t) \right)\, dl\\
        &= \int_{0}^d (a_{-l})_*\left( \int_{\X'}  \delta_{x}  \, d{\mu}(x) \right)\, dl\\
        &= \int_{0}^d (a_{-l})_*\mu \, dl. 
    \end{align*}
    Similarly, we have 
    \begin{align*}
        \left( \int_{\R} t \, d\mu_{\R}(t) \right). \mu_{\X'} &= \int_{\X' \times [c,d]} \int_{-t}^{0} \delta_{a_l x} \, dl d\Tilde{\mu}(x,t)\\
        &\gg \int_{\X' \times [c,d]} \int_{-c}^0 \delta_{a_l x}  \, dl d\Tilde{\mu}(x,t)\\
        &= \int_{0}^c (a_{-l})_*\left( \int_{\X' \times \R}  \delta_{x}  \, d\Tilde{\mu}(x,t) \right)\, dl\\
        &= \int_{0}^c (a_{-l})_*\left( \int_{\X'}  \delta_{x}  \, d{\mu}(x) \right)\, dl\\
        &= \int_{0}^c (a_{-l})_*\mu \, dl. 
    \end{align*}
    This proves the Proposition.
\end{proof}

\begin{proof}[Proof of Theorem \ref{Random Genericity}]
First of all, assume that $x \in \X'$ is Birkhoff generic. We will show that for $(\nu')^{\otimes \N}$ almost every $(b_p)_{p \in \N} \in (G')^{ \N} $, we have that
\begin{align}
    \frac{1}{N} \sum_{p =1}^{N} \delta_{b_p \cdots b_1 x} \rightarrow \mu_{\X'},
\end{align}
as $N \rightarrow \infty$.   
    
    {\bf Case 1.} Assume that $O_e= I_m$ and $O_e' = I_n$ for all $e \in E'$.
    By Breiman's law of large numbers (see e.g. \cite[Sec. 3.2]{BQbook}), we have a co-null subset $\mathcal{A}$ of $(E')^\N$ such that for all $b \in \mathcal{A}$, we have that any weak limit $\mu$ of 
\begin{align}
    \frac{1}{N} \sum_{p =1}^{N} \delta_{b_p \cdots b_1 x},
\end{align}
is $\nu'$-stationary.

Fix $b \in \mathcal{A}$, we claim that \eqref{eq:z2} holds for this $b$. To see this, note that for all $e \in E'$, we have $e= a_{-\log(\rho_e)}$ and due to compactness of $E$, for any $b \in B$, the sequence $(-\log(\rho_{b_i}))$ is bounded in some interval $[c,d]$ with $0< a<b< \infty$. Hence, by Proposition \ref{prop: Sparse Equidistribution} we get that any weak limit of   
\begin{align}
    \frac{1}{N} \sum_{p =1}^{N} \delta_{b_p \cdots b_1 x},
\end{align}
say $\mu$, is a probability measure and
$$
\int_{[0,c]} (a_{-l})_*\mu \, dl \ll \mu_{\X'}.
$$
Since the measure $\mu_{\X'}$ is $\nu$-ergodic and for all $l \in \R$ $(a_{-l})_*\mu$ is $\nu$-stationary, we get that $\mu_{\X'}= (a_{-l})_*\mu$ for a.e. $l$. Thus, $\mu= \mu_{\X'}$. Hence the result holds in this case.\\

{\bf Case 2.} Now assume that $O_e \neq I_m$ or $O_e' \neq I_n$ for some $e \in E$.
Let us define for $b \in B$ and $p \in \N$, the map $t_p: B \rightarrow \R$ as 
\begin{align}
\label{eq:def t n}
t_p(b)= -(\log\rho_{b_1} + \cdots + \log \rho_{b_p}).    
\end{align}
Then by Case 1, we have that for $(\nu')^{\otimes \N}$ almost every $(b_p)_{p \in \N} \in G'^{ \N} $, the following holds
\begin{align}
\label{eq:z3}
    \frac{1}{N} \sum_{p =1}^{N} \delta_{a_{t_p(b)} x} \rightarrow \mu_{\X'},
\end{align}
as $N \rightarrow \infty$. 

We will need the following proposition to proceed.
\begin{prop}[{\cite[Prop. 5.3]{SimmonsWeiss}}]
\label{SimmonsWeiss Prop 5.3}
    Let $Y$ be a locally compact second countable space, $H$ a locally compact second countable group acting continuously on $Y$, $m_Y$ a $H$-invariant and ergodic probability measure on $Y$, and $m_H$ a probability measure on $H$ with compact support $ E_H$. Suppose $\Lambda$ is the subgroup of $H$ generated by $\text{supp}(m_H)$. Let $K$ be a compact subgroup, $m_K$ the Haar measure on $K$, and $\kappa: \Lambda \rightarrow K$ a homomorphism. Let $Z= Y \times K$ and consider the left action of $\Lambda$ on $Z$ defined by the formula $\gamma (x,k) = (\gamma x, \kappa(\gamma)k)$. Assume that this $\Lambda$-action is ergodic with respect to $m_Y \otimes m_K$. Let $\pi_Y: Z \rightarrow Y$ be the projection map onto first factor, and let $m_Z$ be a $m_H$-stationary on $Z$ such that $(\pi_Y)_* m_Z = m_Y$. Then $m_Z = m_Y \otimes m_K$.
\end{prop}

To see this, let $Y= \X'$. Let $H$ be given by
$$
H = \left\{ \begin{pmatrix}
    e^t O \\ & e^{-mt/n} O'
\end{pmatrix}: t \in \R, O \in \text{O}(m), O' \in O(n) \right\}.
$$
The action $\Phi$ of the group $H$ on $\X'$ is given as 
$$
 \Phi  \begin{pmatrix}
    e^t O \\ & e^{-mt/n} I_n
\end{pmatrix} (  x) = \left( \begin{pmatrix}
    e^tI_m \\ & e^{-mt/n} I_n
\end{pmatrix} \right).  x.
$$
The measure $m_Y$ equals $\mu_{\X'}$. Note that the measure $\nu'$ on $\SL_{m+n}(\R)$ has its support in $H$, hence we may treat it as a measure on $H$. Let $m_H$ denote this measure. The support $E_H$ equals $E'$, which is clearly compact. As in the above theorem, we let $\Lambda$ be the subgroup generated by $E'$. The group $K \subset \text{O}(m) \times \text{O}(n)$ is defined as the closure of the group generated by $\{(O_e, O_e')\}_{e \in E}$. The measure $m_K$ is defined as the probability Haar measure on $K$. The map $\kappa: \Lambda \rightarrow K $ is defined as 
$$
\kappa  \begin{pmatrix}
    e^t O \\ & e^{-mt/n} O'
\end{pmatrix} = (O,O').
$$
Note that by the Howe-Moore theorem (see e.g. \cite{ZimmerRobert}), the action of $\Lambda$ on $\X'$ is mixing, and hence also weak mixing. Moreover, the action of $\Lambda$ on $(K, m_K)$ (via $\kappa$) is ergodic since $\kappa(\Lambda)$ is dense in $K$. This implies (see \cite[Prop.~2.2]{KlausSchmidt}) that the product action of $\Lambda$ on $\X' \times K$ is ergodic. 

Now, let $m_Z$ be a weak limit of measures on $\X' \times K$ 
\begin{align}
\label{eq:z5}
\frac{1}{N} \sum_{p=1}^N \delta_{(a_{t_p(b)}, \kappa(b_p \cdots b_1))}    
\end{align}
along a subsequence, say $(N_k)_k$, where $b= (b_1, b_2, \ldots) \in (E')^{\N}$ is chosen from a $ (\nu')^{\otimes \N}$ co-null subset of $(E')^\N$, so that the weak limit is $\nu'$-stationary and \eqref{eq:z3} holds. Clearly this means $(\pi_{\X'})_* m_Z = m_{\X'}$. Hence, by Proposition \ref{SimmonsWeiss Prop 5.3}, we have that $m_Z= m_{\X'} \otimes m_{K}$. Since this holds for any weak limit of \eqref{eq:z5}, we get that for $(\nu')^{\otimes \N}$ a.e. $b \in (E')^{\N}$
\begin{align}
    \label{eq:z4}
    \frac{1}{N} \sum_{p=1}^N \delta_{(a_{t_p(b)}, \kappa(b_p \cdots b_1))} \rightarrow \mu_{\X'} \otimes m_{K}.
\end{align}

Consider the map $\psi: \X' \times K \rightarrow \X'$ as 
$$
\psi(x, (O,O')) = \begin{pmatrix}
    O \\& O'
\end{pmatrix} x.
$$
Clearly, $\psi$ is continuous, we get that for $(\nu')^{\otimes \N}$ a.e. $b \in (E')^{\N}$,
\begin{align*}
     \lim_{N \rightarrow \infty} \frac{1}{N} \sum_{p =1}^{N} \delta_{b_p \cdots b_1 x} &=  \lim_{N \rightarrow \infty} \psi_*(\frac{1}{N} \sum_{p=1}^N \delta_{(a_{t_p(b)}, \kappa(b_p \cdots b_1))}) \\
     &=  \psi_*( \lim_{N \rightarrow \infty} \frac{1}{N} \sum_{p=1}^N \delta_{(a_{t_p(b)}, \kappa(b_p \cdots b_1))}) \\
     &= \psi_*(\mu_{\X'} \otimes m_k) \\
     &=\mu_{\X'}.
\end{align*}
This proves the first part of the Theorem.

We now show the converse. We assume that the point $x \in \X'$ satisfy that for $(\nu')^{\otimes \N}$ almost every $(b_p)_{p \in \N} \in (G')^{ \N} $, we have that
\begin{align}
    \frac{1}{N} \sum_{p =1}^{N} \delta_{b_p \cdots b_1 x} \rightarrow \mu_{\X'},
\end{align}
as $N \rightarrow \infty$.

We will need the following theorem.

\begin{thm}[{\cite[Thm.~2.2]{SimmonsWeiss}}]
\label{SimmonsWeiss Thm2.2}
    Let $H$ be a unimodular connected Lie group, let $\Gamma_H$ be a lattice in $H$, let $X= H / \Gamma_H$, and let $m_X$ be the unique $H$-invariant probability measure on  $X$. Let $\Tilde{\nu}$ be a compactly supported probability measure on $H$. Let $F= \text{supp}(\Tilde{\nu})$. Fix $y \in X$. Suppose that for $\Tilde{\nu}^{\otimes \N}$-a.e. $(h_p)_p \in H^\N$, the sequence $(h_p \cdots h_1y)_{p \in \N}$ is equidistributed with respect to $m_X$. Let $\Lambda$ denote the subgroup of $H$ generated by $F$. Let $K$ be a compact group, let $m_K$ be the Haar measure on $K$, and let $\kappa: \Lambda \rightarrow K$ be homomorphism such that the $\Lambda$-action $\gamma(y,k)= (\gamma y, \kappa(\gamma)k)$ on $X \times K$ is ergodic with respect to $m_X \otimes m_K$. Let $Y$ be a locally compact metric space, $f: F^{\Z} \rightarrow Y$ be a continuous map, and $m_Y= f_* (\nu^{\Z})$.

    Then for $\nu^{\Z}$-a.e. $(h_p)_{p \in \Z}$, the sequence 
    $$
    (h_p \cdots h_1 y, \kappa (h_p \cdots h_1), f(T^p(b)))
    $$
    is equidistributed with respect to the measure $m_X \otimes m_K \otimes m_Y$ on $X \times K \times Y$.
\end{thm}

We will use the above theorem with $H= G'$, $\Gamma_H =\Gamma' $, $X= \X'$, $m_x= \mu_{\X'}$, $\Tilde{\nu}= \nu'$, $y = x$, $F= E'$. The group $K$ is taken to be the closure of subgroup in $O(m) \times O(n) \subset \SL(m+n)(\R)$, generated by $(O_e, O_e')_{e \in E}$. The map $\kappa: \Lambda \rightarrow K$ is the unique homomorphism such that
$$
\kappa(e^{-1})= \begin{pmatrix}
    O_e \\ & O_e'
\end{pmatrix}.
$$
The element $\kappa(e^{-1})$ will also be denoted as $\Tilde{O}_e$.

By the Howe-Moore theorem (see e.g. \cite{ZimmerRobert}), the action of $\Lambda$ on $\X'$ is mixing, and hence also weak mixing. Moreover, the action of $\Lambda$ on $(K, m_K)$ (via $\kappa$) is ergodic since $\kappa(\Lambda)$ is dense in $K$. This implies (see \cite[Prop.~2.2]{KlausSchmidt}) that the product action of $\Lambda$ on $\X' \times K$ is ergodic. 

Let $Y = \R$. Define $f: (E')^\Z \rightarrow \R$ as $f((b_p)_p)= -\log(\rho_{b_0})$. Using Theorem \ref{SimmonsWeiss Thm2.2}, we get a sequence $b=(b_p)_{p \in \Z} \in (E')^\N$ such that 
\begin{align}
    \label{eq:aa1}
    \frac{1}{N} \sum_{p=1}^N \delta_{(b_p \cdots b_1 x,\kappa(b_p \cdots b_1), f(T^p(b)) )} \rightarrow \mu_{\X'} \otimes m_K \otimes f_* (\nu^{\Z}),
\end{align}
as $N \rightarrow \infty$.

Fix $\phi \in \C_c(\X')$. Define $\Phi: \X' \times K \times \R \rightarrow \X'$ as $\Phi(x,k,t)= \int_{l= -t}^0 f( k^{-1}x)$. Also, define $\psi: \X' \times K \times \R \rightarrow \X'$ as $\psi(x,k,t)=t$. Then, we have
\begin{align*}
    & (\int_{\X} \phi \, d\mu_{\X}).(\int_{\R } t \, df_*(\nu^{\otimes \Z})(t) ) \\
    &= \int_{\X \times K \times \R } \int_{l= -t}^{0} \phi(a_l k^{-1} x) \, dl d\mu_{\X} \otimes m_K \otimes f_*(\nu^{\otimes \Z}) (x, k, t) \\
    &= \int_{\X \times K \times \R } \Phi(x,k,t)  \, d\mu_{\X} \otimes m_K \otimes f_*(\nu^{\otimes \Z}) (x, k, t, \av) \\
    &=_{\eqref{eq:aa1}} \lim_{N \rightarrow \infty} \frac{1}{N} \sum_{j=1}^N \Phi(b_j \cdots b_1 x, \kappa(b_j \cdots b_1), f(T^j(\Tilde{b}))) \\
    &= \lim_{N \rightarrow \infty} \frac{1}{N} \sum_{j=1}^N \int_{l= \log \rho_{b_j}}^0 \phi( a_l \Tilde{O}_{b_1} \cdots \Tilde{O}_{b_j}  b_j \cdots b_1 x) \, dm_{\R}(l) \\
    &=  \lim_{N \rightarrow \infty} \frac{1}{N} \sum_{j=1}^N \int_{l= \log \rho_{b_j}}^0 \phi( a_l a_{-(\log \rho_{b_j} + \cdots + \log \rho_{b_1})} x) \, dm_{\R}(l) \\
    &= \lim_{N \rightarrow \infty} \left( \frac{1}{-(\log \rho_{b_N} + \cdots + \log \rho_{b_1})}\int_{l= 0}^{-(\log \rho_{b_N} + \cdots + \log \rho_{b_1})} \phi( a_l x) \, dm_{\R}(l) \right)\\
    &\times \left( \frac{1}{N} \sum_{j=1}^N \psi(b_j \cdots b_1 x, \kappa(b_j \cdots b_1), f(T^j(\Tilde{b}))) \right) \\
    &=_{\eqref{eq:aa1}} \left( \lim_{N \rightarrow \infty} \frac{1}{-(\log \rho_{b_N} + \cdots + \log \rho_{b_1})} \int_{l= 0}^{-(\log \rho_{b_N} + \cdots + \log \rho_{b_1})} \phi( a_l x) \, dm_{\R}(l) \right)\\
    &\times\left( \int_{\X \times K \times \R \times \R^{m+n}} \psi(x, k , t, \av) \, d\mu_{\X} \otimes m_K \otimes f_*(\nu^{\otimes \Z}) (x, k, t, \av) \right)  \\
    &= \left( \lim_{N \rightarrow \infty} \frac{1}{-(\log \rho_{b_N} + \cdots + \log \rho_{b_1})} \int_{l= 0}^{-(\log \rho_{b_N} + \cdots + \log \rho_{b_1})} \phi( a_l x) \, dm_{\R}(l) \right).\left(\int_{\R } t \, df_*(\nu^{\otimes \Z})(t) \right).
\end{align*}
Since this holds for all $\phi \in \C_c(\X)$, we get that 
$$
\lim_{N \rightarrow \infty} \frac{1}{-(\log \rho_{b_N} + \cdots + \log \rho_{b_1})}\int_{l= 0}^{-(\log \rho_{b_N} + \cdots + \log \rho_{b_1})}  \delta_{a_l x} = \mu_{\X}. 
$$
Since $-(\log \rho_{b_N} + \cdots + \log \rho_{b_1}) \rightarrow \infty$ and the gaps $-\log \rho_{b_p}$ ($p \in \N$) are bounded, it follows that 
$$
\lim_{T \rightarrow \infty} \frac{1}{T}\int_{l= 0}^{T}  \delta_{a_l x} = \mu_{\X}.
$$
This proves the converse part of the Theorem.

\end{proof}

\section{Application to Homogeneous Dynamics}
\label{Application to Homogeneous Dynamics}

\subsection{Type 1 measures}
\label{subsec:Type 1 measures}

We begin by introducing the class of sets on which the measures of {\em Type 1} are supported. We follow the notation in section 8.1 of \cite{SimmonsWeiss}.\\

A contracting similarity is a map $\R^m \rightarrow \R^m$ of the form $x \mapsto cO(x)+ y$ where $O$ is an $m \times m$ orthogonal matrix, $c \in (0,1)$, and $y \in \R^m$. A finite similarity IFS on $\R^m$ is a collection of contracting similarities $\Phi= (\phi_e: \R^m \rightarrow \R^m)_{e \in E}$ indexed by a finite set $E$, called the alphabet. Assume that
$$
\phi_e(x)= \rho_e O_e(x) + w_e.
$$
Let $B = E^\N$. The coding map of an IFS $\Phi$ is the map $\eta: B \rightarrow \R^m$ defined by the formula
\begin{align}
\label{eq:def eta}
    \eta(b)= \lim_{k \rightarrow \infty} \phi_{b_p^1}(\alpha_0),
\end{align}
where $\alpha_0 \in \R^m$ is an arbitrary but fixed point and
$$
\phi_{b_p^1}= \phi_{b_1} \circ \cdots \circ \phi_{b_p}.
$$
It is well known that the limit in $\eqref{eq:def eta}$ exists and is independent of the choice of $\alpha_0$ and that the coding map is continuous. The image of $B$ under the coding map called the limit set of $\Phi$, is a compact subset of $\R^m$, which we denote by $\cK= \cK(\Phi)$. 

All of this can be generalized to compact similarity IFSs. In this case, the set $E$ is a compact set and $\Phi= (\phi_e)_{e \in E}$ is a continuously varying family of contracting similarities of $\R^m$. Note that since $E$ is compact,  we have $\sup_{e \in E}\rho_e< 1$. Thus, the coding map $\eta$ is still a continuous map and the image of $B$ under $\eta$ is compact.

 Let $\text{Prob}(E)$ denote the space of probability measures on $E$. For each $\nu \in \text{Prob}(E)$ we can consider the measure $\eta_* \nu^{\otimes \N}$ under the coding map. A measure of the form $\eta_* \nu^{\otimes \N}$ is called a Bernoulli measure. Note that by replacing $E$ by $\text{supp}(\nu)$, we may always assume that $\supp(\nu)=E$. 
 
\begin{defn}
\label{def: Type 1 Measures}
   A measure $\alpha$ on $\R^{m}$ will be called measures of {\em Type 1} if it equals a Bernoulli measure on $\R^m$ (defined as above) such that $\supp(\alpha)$ is not a singleton set. 
\end{defn}

Note that some interesting measures of {\em Type 1} include Lebesgue measures on bounded subsets of lines or hyperplanes in $\R^{m}$.

\subsection{Proof of Theorem \ref{Intermediate Theorem}}
\label{subsec: Proof of Theorem Intermediate Theorem}

Assume the notation as in sub-section \ref{subsec:Type 1 measures}. We fix a fractal $\cK$ and a measure $\nu$ on $E$, such that $\supp(\nu) =E$. 

Now, for each element $e \in E$, define
\begin{align}
    \label{eq:def g_e}
    g_e= \left(\begin{pmatrix}
        \rho^{-1}_e O_e^{-1} \\ & \rho_{e}^{m/n}
    \end{pmatrix}, \begin{pmatrix}
         \rho^{-1}_e O_e^{-1} w_e \\ 0
    \end{pmatrix} \right).
\end{align}
Identifying elements $e \in E$ with $g_e \in G$, we may assume that $E \subset G$, hence $\nu$ is a measure on $G$. 

The main goal of this part of the paper is to prove Theorems \ref{Intermediate Theorem} and \ref{Random Genericity}.

\begin{proof}[Proof of Theorem \ref{Intermediate Theorem}]
    Claim that by Breinman's law of large numbers, Theorem \ref{Main Random walk} and Theorem \ref{Random Genericity}, we have that for $\nu^{\otimes \N}$ almost every $b \in G^\N$, 
    \begin{align}
    \label{eq:y1}
        \frac{1}{N} \sum_{p=1}^{N} \delta_{b_p \cdots b_1 x} \rightarrow \mu_{\X}
    \end{align}
    as $N \rightarrow \infty$. To see this, we use the following form of Breiman's law of large numbers:
    \begin{cor}[{\cite[Cor.~3.4]{BQbook}}]
    \label{cor: BQbook}
        Let $Y$ be a compact metrizable topological space carrying a continuous action of the locally compact second countable group $H$. Suppose $\nu$ is a probability measure on $H$. Then, for any $y$ in $Y$ and for $\nu^\N$-almost any $(b_p)_{p \in \N}$, we have that any weak limit of 
        $$
         \frac{1}{N} \sum_{p=1}^N \delta_{b_{p-1} \cdots b_1 y}
        $$
        is $\nu$-stationary.
    \end{cor}
    We apply corollary \ref{cor: BQbook} to the compact space $Y= \X \cup \{\infty\}$ (the one point compactification of $\X$), with $H=G$ acting usually on $\X$ and trivially on $\infty$. Using Corollary \ref{cor: BQbook}, we have for $\nu^{\otimes \N}$ almost every $b \in G^\N$, any weak limit of LHS of \eqref{eq:y1} is $\nu$-stationary. Thus the same holds for LHS of \eqref{eq:y1}. Now, by Theorem \ref{Random Genericity}, we have that for $\nu^{\otimes \N}$ almost every $b \in G^\N$, the image of any weak limit of LHS of \eqref{eq:y1} under $\pi_*$ equals $\mu_{\X'}$. Thus Theorem \ref{Main Random walk} implies \eqref{eq:y1}.

We will now apply Thm. \ref{SimmonsWeiss Thm2.2} with $H= G$, $\Gamma_H = \Gamma$ and $X = \X$ and $\Tilde{\nu}= \nu$. The point $y$ equals $x \in \X'$, which satisfies the hypothesis of Theorem \ref{SimmonsWeiss Thm2.2} due to \eqref{eq:y1}. 

Let $K \subset \text{O}(d)$ be the closure of the subgroup generated by $(\Tilde{O}_e^{-1})_{e \in E}$ where 
$$
\Tilde{O}_e= \begin{pmatrix}
    O_e \\ & I_n
\end{pmatrix} .
$$
Note that $\Lambda$, the subgroup of $G$ generated by $E$ is contained in the set 
$$
\left\{ \left[\begin{pmatrix}
    e^{t} O \\ & e^{-mt/n} I_n
\end{pmatrix}, v \right]: t \in \R, O \in \text{O}(m) , v \in \R^{m+n} \right\}.
$$
Let $\kappa: \Lambda \rightarrow K $ be defined as 
$$
\left[\begin{pmatrix}
    e^{t} O \\ & e^{-mt/n}
\end{pmatrix}, v \right] \mapsto \begin{pmatrix}
    O \\ & I_n
\end{pmatrix}.
$$
Note that the action of $\Lambda$ on $\X$ is mixing (see for e.g. \cite{Kleinbock}), and hence also weak mixing. Moreover, the action of $\Lambda$ on $(K, m_K)$ (via $\kappa$) is ergodic since $\kappa(\Lambda)$ is dense in $K$. This implies(see \cite[Prop. 2.2]{KlausSchmidt}) that the product action of $\Lambda$ on $\X \times K$ is ergodic. 
Let $Y= \R \times \R^{m+n}$ and $f: E^\Z \rightarrow \R \times \R^{m+n}$ be given by $$f((b_p)_p)= \left(-\log \rho_{b_0}, \begin{pmatrix}
    \eta((b_p)_{p \in \N} \\ 0
\end{pmatrix}\right).$$ Then by Theorem \ref{SimmonsWeiss Thm2.2}, we have that for $\nu^{\Z}$-a.e. $(b_p) \in E^\Z$
\begin{align}
    \label{eq:y3}
    \lim_{N \rightarrow \infty} \frac{1}{N} \sum_{p=1}^N \delta_{(b_p \cdots b_1 x, \kappa(b_p \cdots b_1), f(T^p(b)))} \rightarrow \mu_{\X'} \otimes m_K \otimes f_*(\nu^{\otimes \Z}).
\end{align}

Note that for any $b \in E^\N$, we have 
\begin{align}
    \label{eq:y4}
    a_{t_j(b)} [I_{m+n}, \begin{pmatrix}
     \eta(b)    \\0
    \end{pmatrix}] x &= \Tilde{O}_{b_1} \cdots \Tilde{O}_{b_j} [I_{m+n}, \begin{pmatrix}
     \eta(T^j(b))    \\0
    \end{pmatrix}] b_j \cdots b_1x.
\end{align}

Define $\psi: \X \times K \times \R \times \R^{m+n} \rightarrow \R$ as $(x,k, t, \av) \mapsto t$.
Fix $\phi \in \C_c(\X)$. Define $\Phi:\X \times K \times \R \times \R^{m+n} $ as 
$$
\Phi(x, k, t, \av) = \int_{l= -t}^0 \phi(a_l k^{-1} [I_{m+n}, \av] x) \, dm_{\R}(l).
$$
Clearly $\Phi$ is continuous and $\Phi \in L_\infty(\X \times K \times \R \times \R^{m+n},\mu_{\X} \otimes m_K \otimes f_*(\nu^{\otimes \Z}) )$. Thus, we have for any sequence $\Tilde{b}=(b_p)_{p \in \Z}$ (with $b= (b_p)_{p \in \N}$) satisfying \eqref{eq:y3}
\begin{align*}
    & \left(\int_{\X} \phi \, d\mu_{\X}\right).\left(\int_{\R \times \R^{m+n}} t \, df_*(\nu^{\otimes \Z})(t, \av) \right) \\
    &= \int_{\X \times K \times \R \times \R^{m+n} } \int_{l= -t}^{0} \phi(a_l k^{-1} [I_{m+n}, \av] x) \, dl d\mu_{\X} \otimes m_K \otimes f_*(\nu^{\otimes \Z}) (x, k, t, \av) \\
    &= \int_{\X \times K \times \R \times \R^{m+n} } \Phi(x,k,t, \av)  \, d\mu_{\X} \otimes m_K \otimes f_*(\nu^{\otimes \Z}) (x, k, t, \av) \\
    &=_{\eqref{eq:y3}} \lim_{N \rightarrow \infty} \frac{1}{N} \sum_{j=1}^N \Phi(b_j \cdots b_1 x, \kappa(b_j \cdots b_1), f(T^j(\Tilde{b}))) \\
    &= \lim_{N \rightarrow \infty} \frac{1}{N} \sum_{j=1}^N \int_{l= \log \rho_{b_j}}^0 \phi( a_l \Tilde{O}_{b_1} \cdots \Tilde{O}_{b_j} [I_{m+n},\begin{pmatrix}
        \eta(T^j(b)) \\ 0
    \end{pmatrix} ] b_j \cdots b_1 x) \, dm_{\R}(l) \\
    &=_{\eqref{eq:y4}}  \lim_{N \rightarrow \infty} \frac{1}{N} \sum_{j=1}^N \int_{l= \log \rho_{b_j}}^0 \phi( a_l a_{-(\log \rho_{b_j} + \cdots + \log \rho_{b_1})} [I_{m+n},\begin{pmatrix}
        \eta(b) \\ 0
    \end{pmatrix} ] x) \, dm_{\R}(l) \\
    &= \lim_{N \rightarrow \infty} \left( \frac{1}{t_N(b)}\int_{l= 0}^{t_N(b)} \phi( a_l [I_{m+n},\begin{pmatrix}
        \eta(b) \\ 0
    \end{pmatrix} ] x) \, dm_{\R}(l) \right) \left( \frac{1}{N} \sum_{j=1}^N \psi(b_j \cdots b_1 x, \kappa(b_j \cdots b_1), f(T^j(\Tilde{b}))) \right) \\
    &=_{\eqref{eq:y3}} \left( \lim_{N \rightarrow \infty} \frac{1}{t_N(b)} \int_{l= 0}^{t_N(b)} \phi( a_l[I_{m+n},\begin{pmatrix}
        \eta(b) \\ 0
    \end{pmatrix} ] x) \, dm_{\R}(l) \right)\\
    &\times \left( \int_{\X \times K \times \R \times \R^{m+n}} \psi(x, k , t, \av) \, d\mu_{\X} \otimes m_K \otimes f_*(\nu^{\otimes \Z}) (x, k, t, \av) \right)  \\
    &= \left( \lim_{N \rightarrow \infty} \frac{1}{t_N(b)} \int_{l= 0}^{t_N(b)} \phi( a_l [I_{m+n},\begin{pmatrix}
        \eta(b) \\ 0
    \end{pmatrix} ] x) \, dm_{\R}(l) \right).\left(\int_{\R \times \R^{m+n}} t \, df_*(\nu^{\otimes \Z})(t, \av) \right).
\end{align*}
Since this holds for all $\phi \in \C_c(\X)$, we get that 
$$
\lim_{N \rightarrow \infty} \frac{1}{t_N(b)}\int_{l= 0}^{t_N(b)}  \delta_{a_l [I_{m+n},\begin{pmatrix}
        \eta(b) \\ 0
    \end{pmatrix} ] x} = \mu_{\X}. 
$$
Since $t_N(b) \rightarrow \infty$ and the gaps $t_{N+1}(b)- t_N(b)$ ($N \in \N$) are bounded, it follows that 
$$
\lim_{T \rightarrow \infty} \frac{1}{T}\int_{l= 0}^{T}  \delta_{a_l [I_{m+n},\begin{pmatrix}
        \eta(b) \\ 0
    \end{pmatrix} ] x} = \mu_{\X}.
$$
This proves the Theorem.
\end{proof}

\section{Diophantine Approximation}
\label{Diophantine Approximation}

In this section, we will study inhomogeneous Diophantine properties of $(A, b)$, when the Diophantine properties of $A$ are known and $b$ varies in a fractal. 

\begin{defn}
    Define a map $u: \Mat \rightarrow \SL_{m+n}(\R)$ as 
    $$
    u(A)= \begin{pmatrix}
        I_m & A \\ & I_n
    \end{pmatrix}.
    $$
\end{defn}

\subsection{Type 2 Measures}
\label{subsec: Type 2 measures}
Following \cite{ABRD} and \cite{solanwieser}, we define constraining pencils. Let $P< G'$ be the parabolic group given by
\begin{align}\label{eq: parabolic}
P = \{g \in G: \lim_{t \to \infty}a_tg a_{-t} \text{ exists}\}
\end{align}
and note that we may identify $P\backslash G \simeq \Gr(m, m+n)$.

\begin{defn}
For an integer $r \leq m$ and a proper subspace $W \subset \R^{m+n}$ the pencil $\mathfrak{P}_{W, r}$ is given by:
\begin{align*}
\mathfrak{P}_{W, r} = \{V \in \Gr(m, m+n): \dim(V\cap W) \geq r\}.
\end{align*}
The pencil is constraining if $\frac{\dim(W)}{r} < \frac{m+n}{m}$ and weakly constraining if  $\frac{\dim(W)}{r} \leq \frac{m+n}{m}$.
\end{defn}

\begin{defn}
A partial flag subvariety of $P\backslash G$ (for $\{a_t\}$) is a subvariety of the form $PHg$ where $H < G$ is a reductive group containing $\{a_t\}$ and where $g \in G$.
\end{defn}

A measure $\beta$ on $\Mat$ is said to be a {\em Type 2 measure} if it equals the pushforward of Lebesgue measure on $[0,1]$ under an analytic map $\varphi: [0,1] \rightarrow \R^m$, such that the image of $u \circ \phi$ in $P \backslash G$ is not contained in a weakly constraining pencil or in a partial flag subvariety.

\begin{rem}
    In case $m=1$, the above conditions for $\varphi: [0,1] \to \Mat$ translate to the image of $\varphi$ not being contained in any affine hyperplane.
\end{rem}

\begin{thm}[{\cite[Thm.~1.10]{solanwieser}}]
\label{thm: genericity of Type 2 measures}
    For $\beta$ a measure of {\em Type 2}, the following holds: for any $x \in \X'$ and for $\beta$-a.e. $A\in \Mat$, we have $u(A)x$ is Birkhoff generic.
\end{thm}

\subsection{Type 3 Measures}
\label{subsec: Type 3 measures}
Following \cite{SimmonsWeiss}, we define an \emph{algebraic similarity} of $\Mat$ to be a map of the form $A \mapsto \rho O_1 A O_2 + B$, where $O_1 \in O(m)$, $O_2 \in O(n)$, $\rho > 0$ and $B \in \Mat$. If $m = 1$ or $n = 1$, then every similarity is algebraic.  An algebraic similarity IFS is a continuously varying family of algebraic similarities of $\Mat$. It will be called irreducible if it does not leave invariant any proper affine subspace of $\Mat \cong \R^{m n}$. It will be called compact if the set $E$ of algebraic similarities is a compact set and $\Phi= (\phi_e)_{e \in E}$ is continuously varying family of algebraic similarities of $\Mat$.\\ 

A measure $\mu\in\Prob(E)$ is called contracting on average if 
\[
\int \log\|\phi_e'\| \, d\mu(e) < 0,
\]
where $\|\phi_e'\|$ denotes the scaling constant of the similarity
$\phi_e$ which is equal to the norm of the derivative $\phi_e'$ at any point
of $\R^d$. If $\mu$ is contracting on average, then by the ergodic
theorem $\|\phi_{b^1_p}'\| \to 0$ exponentially fast for $\beta$-a.e. $b\in E^{\N}$. In this case, the limit 
\begin{align}
\label{eq:def eta 2}
    \eta(b)= \lim_{k \rightarrow \infty} \phi_{b_p^1}(\alpha_0),
\end{align}
converges a.e. $b \in E^{\N}$ and for any fixed $\alpha_0 \in \Mat$, thereby defining a measure-preserving map
$\pi:(E^\N,\nu^{\otimes}) \to (\R^d,\eta_*(\nu)^{\otimes \N})$. Note that the definition of $\eta$ is independent of the choice of $\alpha_0$.

\begin{defn}\label{def: type 3 measures}
Let $\Phi$ be an irreducible compact algebraic similarity IFS on $\Mat$ and fix $\nu\in\Prob(E)$, contracting on average, such that $\supp(\nu) = E$. Then the Bernoulli measure $\eta_*(\nu)^{\otimes \N}$ consists of {\em Type 3 measures}. In addition, if there exists $e \in E$ such that $\phi_e(0)=0$ then we will call it a measure of {\em Type 3'}. 
\end{defn}

\begin{thm}[{\cite[Thm.~8.11]{SimmonsWeiss}}]
\label{thm: genericity of Type 3 measures}
     For $\gamma$ a measure of {\em Type 3}, the following holds: for $\gamma$-a.e. $A\in \Mat$, we have $u(A)x$ is Birkhoff generic.
\end{thm}

Let $\Phi$ be an irreducible compact algebraic similarity IFS on $\Mat$ and fix $\nu\in\Prob(E)$, contracting on average, such that $\supp(\nu) = E$. Assume that for some $e \in E$, we have $\phi_e(0)= 0$. Let $\gamma = \eta_*(\nu){\otimes \N}$. Suppose for all $e \in E$, we have
$$
\phi_e(A)= \rho_e O_e A O_e' + B_e.  
$$
Define for all $e \in E$, the element $g_e \in G'$ as 
$$
g_e = \begin{pmatrix}
    \rho_e^{-n/m+n}O_e^{-1} & - \rho_e^{-n/m+n}O_e^{-1}B_e  \\ & \rho_e^{m/m+n} O_e'
\end{pmatrix}.
$$
Identifying $e$ with $g_e$, we may assume that $E \subset G'$. Then $\nu$ can be treated as a measure on $G$. It was proved in \cite[Section 10.1]{SimmonsWeiss} that the Zariski closure of the closed subgroup generated by $\supp(\mu)$ contains the group $u(\Mat)$. Also, it is easy to see that
$$
b_p \cdots b_1= \begin{pmatrix}
    O_{b_p}^{-1} \cdots O_{b_1}^{-1} \\ &O_{b_p}^{'} \cdots O_{b_1}^{'}
\end{pmatrix} \begin{pmatrix}
    (\rho_{b_p} \cdots \rho_{b_1})^{-n/m+n}I_m \\ & (\rho_{b_p} \cdots \rho_{b_1})^{-n/m+n}I_n
\end{pmatrix} 
$$
$$\times u(\phi_{b_p} \circ \cdots \circ \phi_{b_1}(0)).$$
The map $(b_p)_p \mapsto \lim_{p \rightarrow \infty} \phi_{b_p} \circ \cdots \circ \phi_{b_1}(0)$ therefore equals $\eta$. Thus, the conditions of \cite[Thm.~1.11]{ProhaskaSertShi} are satisfied and it yields the following result in our special case:

\begin{thm}
\label{thm: genericity of Type 3' measures}
    Fix $\gamma$ a measure of {\em Type 3'}. Then, for any $x \in \X'$ and for $\gamma$-a.e. $A\in \Mat$, we have 
    $$
    \frac{1}{T} \int_{0}^T \delta_{a_tu(A) x} \rightarrow \mu_{\Bar{G'.x}},
    $$
    where $\Bar{G'.x}$ denote the closure of $G'$ orbit of $x$. Note that by Ratner's Theorem, we know that $\Bar{G'.x}$ is a homogeneous space for any $x \in \X$. The measure $\mu_{\Bar{G'.x}}$ denotes the homogeneous probability measure on $\Bar{G'.x}$.
\end{thm}

\subsection{Proof of Theorem \ref{thm:Single Inhomogeneous case 1}}
\begin{proof}[Proof of Theorem \ref{thm:Single Inhomogeneous case 1}]
    Let $\mathcal{M}= \{A \in \Mat: u(A)\Gamma' \text{ is Birkhoff generic}\}$. Clearly, by Theorems \ref{thm: genericity of Type 2 measures} and \ref{thm: genericity of Type 3 measures}, we have that $\mathcal{M}$ gets full measure with respect to each measure of {\em Type 2} and {Type 3}.

    Also, it is clear that if for some $A$ we have $\GenA \neq \emptyset$, then we have for $b \in \GenA$ ,
    \begin{align}
    \label{eq:s1}
         \frac{1}{T} \int_{0}^T \delta_{a_l [u(A), \Tilde{b}]} \rightarrow \mu_{\X},
    \end{align}
    where $\Tilde{b} \in \R^{m+n}$ equals $\begin{pmatrix}
        b \\ 0
    \end{pmatrix}$. Applying the map $\pi_*$ to \eqref{eq:s1}, we get that $ u(A)\Gamma' \text{ is Birkhoff generic}$, i.e, $A \in \mathcal{M}$. This proves the last claim.

   Fix $A \in \mathcal{M}$. The fact that $\alpha(\GenA)= 1$ for all measures of {\em Type 1} follows from Theorem \ref{Intermediate Theorem}. Now using the Dani correspondence, we get from $\alpha(\GenA)= 1$, that $\alpha(\BadA)= \alpha(\DirichletA) = 0$. This proves the Theorem.
\end{proof}

\subsection{Proof of Theorem \ref{thm:Single Inhomogeneous case 2}}
\begin{proof}[Proof of Theorem \ref{thm:Single Inhomogeneous case 2}]
Fix $\gamma$ a measure of {\em Type 3} and let $b \in \R^m$. Let $\Tilde{b}= \begin{pmatrix}
        b \\ 0
    \end{pmatrix}$.
    Firstly, assume that $b \notin \Q^m$. Let $\Tilde{b}= \begin{pmatrix}
        b \\ 0
    \end{pmatrix}$. Then, by Theorem \ref{thm: genericity of Type 3' measures} we get that for $\gamma$-a.e. $A \in \Mat$, the element $[u(A), \Tilde{b}]\Gamma$ is Birkhoff generic. Thus, $\gamma(\Genb)= 1$. By Dani's correspondence, we get that $\gamma(\Badb)= \gamma(\Dirichletb)= 0$.

    Now, assume that $b \in \Q^m$. Again, using Theorem \ref{thm: genericity of Type 3' measures} we get that for $\gamma$-a.e. $A \in \Mat$, the closure of set $ a_t[u(A), \Tilde{b}]\Gamma$ is $\Bar{G'.[I_{m+n}, \Tilde{b}]}$. Again, by Dani's correspondence, we get that $\gamma(\Badb)= \gamma(\Dirichletb)= 0$. 
\end{proof}

\bibliography{Biblio}
\end{document}